\documentclass[12pt,oneside,reqno]{amsart}
\usepackage[english]{babel}
\usepackage{indentfirst}
\usepackage{eucal}
\usepackage[utf8]{inputenc}
\usepackage{paralist}
\usepackage{amssymb}
\usepackage{amsthm}
\usepackage{amsmath}
\usepackage{amscd}
\usepackage{graphicx}
\usepackage{hyperref}
\usepackage[margin=1.2in]{geometry}
\usepackage{cite}
\usepackage{lineno} 
\usepackage[alphabetic]{amsrefs}
\allowdisplaybreaks

\theoremstyle{plain} 
\newtheorem{theorem}{Theorem}[section] 
\newtheorem{corollary}[theorem]{Corollary} 
\newtheorem{lemma}[theorem]{Lemma}

\newtheorem{defn}[theorem]{Definition}
\newtheorem{open}[theorem]{Open Problem}

\theoremstyle{definition}

\theoremstyle{remark}

\numberwithin{equation}{section} 

\newcommand{\R}{\ensuremath{\mathbb{R}}}

\newcommand{\N}{\ensuremath{\mathbb{N}}}
\newcommand{\C}{\ensuremath{\mathbb{C}}}

\renewcommand{\epsilon}{\varepsilon}
\newcommand{\e}{\varepsilon}

\renewcommand{\leq}{\leqslant}
\renewcommand{\le}{\leqslant}

\renewcommand{\ge}{\geqslant}

\begin{document}


\title[Flexibility of linear equations and
rigidity of minimal graphs]{Boundary properties of fractional objects:\\
flexibility of linear equations\\
and rigidity of minimal graphs}

\author[Serena Dipierro,
Ovidiu Savin, and Enrico Valdinoci]{Serena Dipierro${}^{(1)}$\and
Ovidiu Savin${}^{(2)}$
\and
Enrico Valdinoci${}^{(1)}$}

\maketitle

{\scriptsize \begin{center} (1) -- Department of Mathematics and Statistics\\
University of Western Australia\\ 35 Stirling Highway, WA6009 Crawley (Australia)\\
\end{center}
\scriptsize \begin{center} (2) --
Department of Mathematics\\
Columbia University\\
2990 Broadway, NY 10027
New York (USA)
\end{center}
\bigskip

\begin{center}
E-mail addresses:
{\tt serena.dipierro@uwa.edu.au},\\
{\tt savin@math.columbia.edu},\\
{\tt enrico.valdinoci@uwa.edu.au}
\end{center}
}
\bigskip\bigskip

\begin{abstract}
The main goal of this article is to understand the trace properties
of nonlocal minimal graphs in~$\R^3$, i.e. nonlocal minimal surfaces with a graphical structure.

We establish that at any boundary points at which 
the trace from inside happens to coincide with
the exterior datum, also the tangent planes
of the traces necessarily coincide with those of the exterior datum.

This very rigid geometric constraint is in sharp contrast with the case
of the solutions of the linear equations driven by the fractional Laplacian,
since we also show that, in this case, the fractional normal
derivative can be prescribed arbitrarily, up to a small error.

We remark that, at a formal level,
the linearization of the trace of a nonlocal minimal graph
is given by the fractional normal derivative of a fractional Laplace problem,
therefore the two problems are formally related. Nevertheless, the nonlinear equations of fractional mean curvature type
present very specific properties which are strikingly different from
those of other problems of fractional type
which are apparently similar, but
diverse in structure, and the nonlinear
case given by the nonlocal minimal graphs
turns out to be significantly more rigid than its linear counterpart.
\end{abstract}

\section{Introduction}

\subsection{Boundary behavior of fractional objects}

This article investigates the geometric properties
at the boundary of solutions of fractional problems.
Two similar, but structurally significantly different, situations are taken
into account. On the one hand, we will consider the solution
of the {\em linear fractional equation}
\begin{equation*} \begin{cases}
(-\Delta)^\sigma u=0 & {\mbox{ in }} B_1\cap \{x_n>0\},\\
u=0 & {\mbox{ in }} \{x_n<0\},\end{cases}\end{equation*}
where~$\sigma\in(0,1)$, and, for~$x'=(x_1,\dots,x_{n-1})$
with~$|x'|<1$, we consider the ``fractional boundary derivative''
\begin{equation}\label{56:62734899fi02}
\displaystyle\lim_{x_n\searrow0} \frac{u(x',x_n)}{x_n^\sigma}
.\end{equation}
Interestingly, the function in~\eqref{56:62734899fi02}
plays an important role in understanding fractional equations,
see~\cite{MR3168912}. In particular,
while classical elliptic equations are smooth up to the boundary,
the solutions of fractional equations with prescribed exterior
datum are in general not better
than H\"older continuous with exponent~$\sigma$,
and therefore the function in~\eqref{56:62734899fi02}
is the crucial ingredient to detect the growth of the solution in the vicinity
of the boundary.

As a first result, we will show here that, roughly speaking,
the function in~\eqref{56:62734899fi02} can be {\em arbitrarily prescribed,
up to an arbitrarily small error}. That is, one can construct solutions
of linear fractional equations whose fractional boundary derivative
behaves in an essentially arbitrary way.\medskip

Then, we turn our attention to the boundary property of {\em nonlocal minimal
graphs}, i.e. minimizers of the fractional perimeter functional which possess
a graphical structure. In this case, we show that the boundary properties
are subject to severe geometric constraints, in sharp contrast with
the case of fractional equations.

First of all, the continuity properties of nonlocal minimal graphs are very
different from those of the solutions of fractional equations, since we have
established in~\cites{MR3516886, MR3596708} that
nonlocal minimal graphs are not necessarily continuous at the boundary.
In addition, the boundary discontinuity of nonlocal minimal graphs in the plane
happens to be a ``generic'' situation, as we have recently proved in~\cite{2019arXiv190405393D}.

The focus of this article is on the three-dimensional setting,
i.e. the case
in which the graph is embedded in~$\R^3$.
In this situation, the graph can be continuous at a given point,
but the discontinuity may occur along the trace, at nearby points.
More precisely, we will look at a function~$u:\R^2\to\R$, which is an $s$-minimal graph in~$(-2,2)\times(0,4)$
and is such that~$u=0$ in~$ (-2,2)\times(-h,0)$, for some~$h>0$.
In this setting, we will consider its trace along~$\{x_2=0\}$, namely we consider
the function
\begin{equation*}
\lim_{x_2\searrow0}u(x_1,x_2).
\end{equation*}
The main question that we address in this article is
precisely
whether or not {\em the trace of a nonlocal minimal graph possesses
any distinctive feature or satisfies any particular geometric
constraint}. 

We will prove that, differently from the case of the linear equations
(and also in sharp contrast to the case of classical minimal surfaces),
the traces of nonlocal minimal graphs {\em cannot
have an arbitrary shape}, and, in fact, {\em matching points from the
two sides
must necessarily occur with horizontal tangencies}.
\medskip

This result relies on a classification theory for homogeneous
graphs, since we will show that, in this case, {\em the matching at the origin
is sufficient to make a nonlocal minimal graph trivial}.
\medskip

In the rest of this introduction, we will present the precise
mathematical framework in which we work and provide the formal statements
of our main results.

\subsection{Boundary flexibility of linear fractional equations}

We discuss now the case of fractional linear equations,
showing that the fractional boundary derivative of the solutions
can be essentially arbitrarily prescribed, up to a small error.
For this,
we denote by~$B'_r$ the $(n-1)$-dimensional
ball of radius~$r$ centered at the origin, namely
$$ B'_r := \{ x'\in\R^{n-1} {\mbox{ s.t. }} |x'|<r\}.$$
As customary, given~$\sigma\in(0,1)$, we define the fractional Laplacian as
$$ (-\Delta)^\sigma u(x):=\int_{\R^n}\frac{2u(x)-u(x+y)-u(x-y)}{|y|^{n+\sigma}}
\,dy.$$
Then, we have:

\begin{theorem}\label{FLFL}
Let~$n\ge2$, $\sigma\in(0,1)$, $k\in\N$ and~$f\in C^k(\overline{B_1'})$. Then, for every~$\e>0$
there exist~$f_\e\in C^k(\R^{n-1})$ and~$u_\e\in C(\R^n)$ such that
\begin{equation}\label{56:01} \begin{cases}
(-\Delta)^\sigma u_\e=0 & {\mbox{ in }} B_1\cap \{x_n>0\},\\
u_\e=0 & {\mbox{ in }} \{x_n<0\},\end{cases}\end{equation}
\begin{equation}\label{56:02}
\displaystyle\lim_{x_n\searrow0} \frac{u_\e(x',x_n)}{x_n^\sigma}=f_\e(x')\qquad{\mbox{ for all }}
x'\in B_1',\end{equation}
and
\begin{equation}\label{56:03}
\| f_\e-f\|_{C^k(\overline{B_1'})}\le\e.
\end{equation}
\end{theorem}

On the one hand, Theorem~\ref{FLFL} falls in the research line
opened in~\cite{MR3626547}
according to which ``all functions are $\sigma$-harmonic up to a small error'',
namely it states an interesting flexibility offered by solutions
of fractional equations which can adapt themselves
in order to capture essentially any prescribed behavior.
This flexible feature has been recently studied
in several fractional contexts, including time-fractional derivatives,
non-elliptic operators, and higher order operators, see~\cites{MR3716924, MR3935264, KRYL, CAR, CARBOO}.
In addition, the flexible properties of fractional equations
can be effectively exploited to construct interesting counterexamples,
see~\cite{MR3783214}, and they have consequences
in concrete scenarios involving also
mathematical biology and inverse problems, see~\cites{MR3579567, MR3774704}.
Differently from the previous literature, Theorem~\ref{FLFL}
aims at detecting a ``boundary'' flexibility of fractional equations,
rather than an ``interior'' one.

On the other hand, differently from all the other fractional flexibility results
in the literature, which have no counterpart for the case of the classical Laplacian,
Theorem~\ref{FLFL} shares a common treat with the Laplace equation
and possesses a full classical analogue (we present its
classical counterpart in Appendix~\ref{APPA}).

\subsection{Boundary rigidity of fractional minimal surfaces}

We now discuss the boundary behavior of $s$-minimal
surfaces and we will show its striking differences
with respect to the linear fractional equations.
To this end, we recall the setting introduced in~\cite{MR2675483}.
Given~$s\in(0,1)$, we consider the interaction
of two disjoint (measurable) sets~$F$, $G\subseteq\R^N$ defined by
\begin{equation} \label{nNdi}
{\mathcal{I}}_s(F,G):=\iint_{F\times G}\frac{dx\,dy}{|x-y|^{N+s}}.\end{equation}
Given a bounded reference domain~$\Omega$ with Lipschitz boundary,
we define the $s$-perimeter of a set~$E\subseteq\R^N$ in~$\Omega$ by
$$ {\rm Per}_s(E;\Omega):=
{\mathcal{I}}_s(E\cap\Omega,E^c\cap\Omega)+
{\mathcal{I}}_s(E\cap\Omega,E^c\cap\Omega^c)+
{\mathcal{I}}_s(E\cap\Omega^c,E^c\cap\Omega).$$
As customary, we have used here the complementary set
notation~$E^c:=\R^N\setminus E$.

\begin{defn}
Let~$\Omega\subset\R^N$ be bounded and with Lipschitz boundary.
Let~$E\subseteq\R^N$.
We say that~$E$ is~$s$-minimal in~$\Omega$ if~${\rm Per}_s(E;\Omega)<+\infty$ and
\begin{equation}\label{MINIMIZZ} {\rm Per}_s(E;\Omega)\le {\rm Per}_s(E';\Omega)\end{equation}
for every~$E'\subseteq\R^N$ such that~$E'\cap\Omega^c=E\cap\Omega^c$.

Moreover, given~$U\subset\R^N$,
we say that~$E$ is locally~$s$-minimal in~$U$ 
if it $s$-minimal in~$
\Omega$, for every~$\Omega$
which is bounded, with Lipschitz boundary
and strictly contained in~$U$.
\end{defn}

We remark that one can make sense of
the minimization procedure also in unbounded domains by saying that~$E$ is
locally~$s$-minimal in a (possibly unbounded) domain~$\Omega$
if~$E$ is $s$-minimal in every bounded and Lipschitz domain~$\Omega'\Subset\Omega$
(see Section~1.3 in~\cite{MR3827804} for additional details
on these minimality notions).
\medskip 

The regularity theory of nonlocal minimal surfaces
is a fascinating topic of investigation, still possessing a number
of fundamental open problems. We refer to~\cites{MR3090533, MR3107529, MR3331523, BV} for interior regularity results,
\cites{MR3798717, CCC} for a precise discussion on stable nonlocal cones,
and~\cites{MR3588123, MR3824212} for recent surveys containing
the state of the art of this problem.
\medskip

A particularly important case of locally $s$-minimal
sets is given by the ones which have a graph structure.
Namely, given~$\Omega_0\subseteq\R^n$ and~$u:\Omega_0\to\R$,
we let
\begin{equation}\label{LEYUDEF} E_u:=\{ X=(x,x_{n+1}) {\mbox{ s.t. $x\in\Omega_0$ and $ x_{n+1}<u(x)$}}\}.\end{equation}
With respect to the notation in~\eqref{nNdi},
we are taking here~$N=n+1$.

\begin{defn}
We say that~$u$ is an $s$-minimal graph in~$\Omega_0$
if~$E_u$ is a locally $s$-minimal set in~$\Omega_0\times\R$.
\end{defn}

Interestingly, $s$-minimal graphs enjoy suitable Bernstein-type
properties, see~\cites{MR3680376, PISA, 2018arXiv180705774C},
and they have a
smooth interior regularity theory, as proved in~\cite{MR3934589}.

See also~\cite{NOCHETTO} for several very precise simulations on nonlocal minimal graphs
and a sharp numerical analysis of their properties.
\medskip

With this, we are ready to state the main result of this article,
which gives some precise geometric conditions on the trace of nonlocal minimal
graphs. We establish that the trace graph has necessarily {\em zero derivatives
when the trace crosses zero}. The precise result that we have is the following one:

\begin{theorem}\label{NLMS}
Let~$u$ be an $s$-minimal graph in~$(-2,2)\times(0,4)$.
Assume that there exists~$h>0$
such that~$u=0$ in~$ (-2,2)\times(-h,0)$,
and let
\begin{equation}\label{MR3516886EQ}
f(x_1):=\lim_{x_2\searrow0}u(x_1,x_2).
\end{equation}
Then, there exist~$\delta_0\in\left(0,\frac1{100}\right)$
and~$C>0$
such that if~$\zeta_0\in\left(-\frac{3}2,\frac32\right)$ is such that~$f(\zeta_0)=0$, then
\begin{equation}\label{VANP} |u(x)|\le C\,|x-(\zeta_0,0)|^{\frac{3+s}2}\end{equation}
for every~$x\in B_{\delta_0}(\zeta_0,0)$, and, in particular,
\begin{equation}\label{THAN-2}
{\mbox{$f'(\zeta_0)=0$.}}\end{equation}
\end{theorem}

We remark that the existence of the limit
in~\eqref{MR3516886EQ}
is warranted by Theorem~1.1 in~\cite{MR3516886}.\medskip

\begin{figure}
    \centering
    \includegraphics[width=8.9cm]{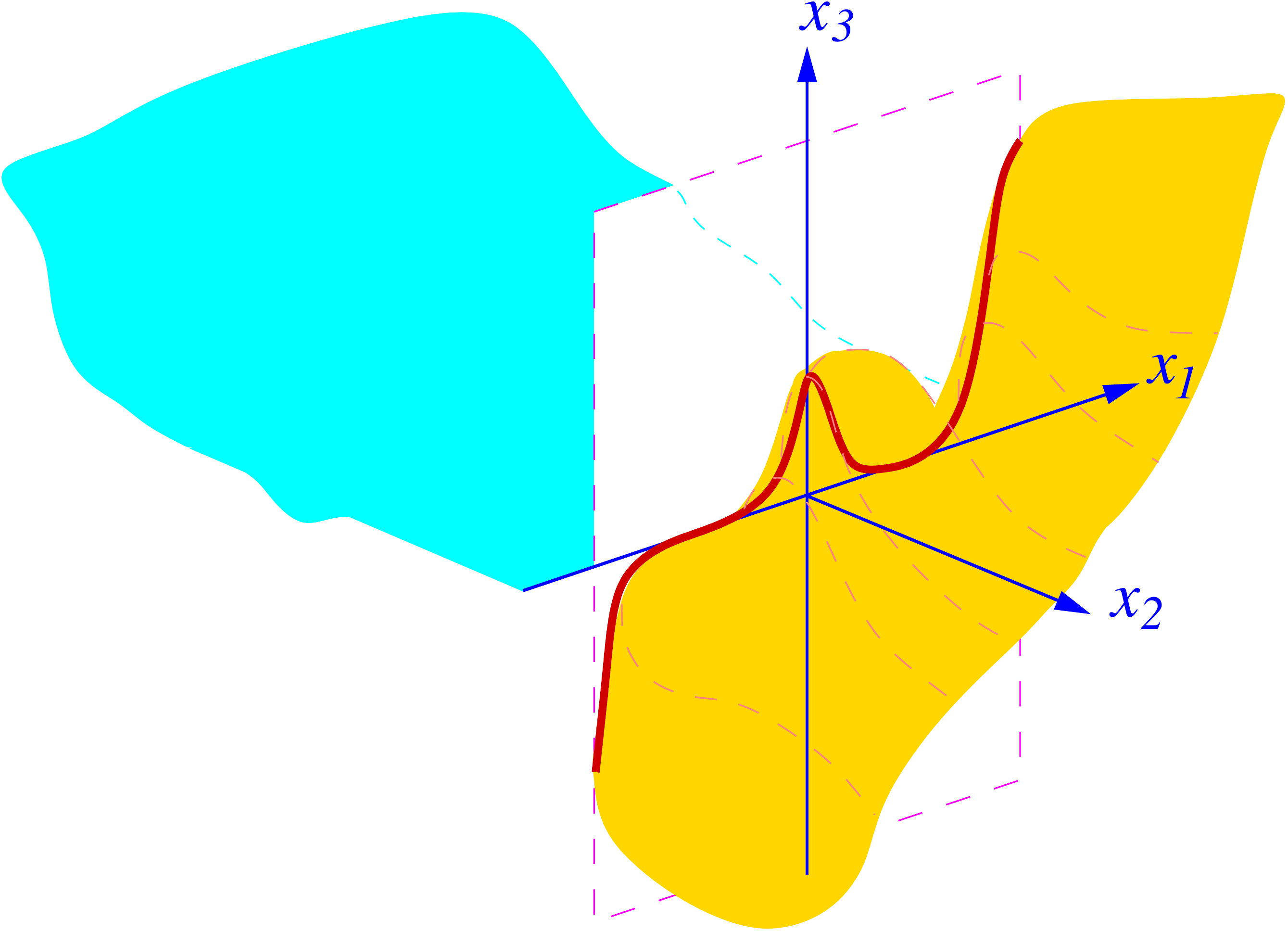}
    \caption{\em {{A nonlocal minimal graph, in the light of Theorem~\ref{NLMS}. The red curve represents
    the graph of the function~$f$ defined in~\eqref{MR3516886EQ}.}}}
    \label{DOMA}
\end{figure}

The statement of Theorem~\ref{NLMS} is described\footnote{
We observe that drawing the trace
along~$\{x_2=0\}$ 
in Figure~\ref{DOMA} in a communicative way
is not completely easy, since
the vertical tangencies make the nonlocal
minimal surface in~$\{x_2>0\}$ hide its own trace.} in Figures~\ref{DOMA}
and~\ref{DOMA-2}. With respect
to this, we stress the remarkable geometric property
given by the {\em vanishing of the gradient of the trace
at the
zero crossing points}. We observe that this situation is completely
different with respect to the one arising for solutions of linear equations,
and one can compare the structurally ``rigid'' geometry imposed by
Theorem~\ref{NLMS} with the almost completely arbitrariness
arising in Theorem~\ref{FLFL}.\medskip

We remark that, at a formal level,
the settings in Theorems~\ref{FLFL} and~\ref{NLMS}
are strictly related, since
the linearization of the trace of a nonlocal minimal graph
is given by the fractional normal derivative of a fractional Laplace problem.
More specifically, 
when one takes into account the improvement of flatness argument
for an $\e$-flat nonlocal
minimal graph~$u$
(see the forthcoming Lemma~\ref{NOANpierjfppp34}),
one sees that~$u/\e$ shadows a function~$\bar u$,
which is a solution of~$(-\Delta)^{\sigma}\bar u=0$ in~$\{x_2>0\}$,
with~$\sigma:=\frac{1+s}2$: in this context the first order of~$\bar u$
near the origin takes the form~$\bar a \,x_2^{\sigma}$, for some~$a\in\R$.
Comparing with~\eqref{56:02}, one has that~$\bar a$ is exactly the
fractional normal derivative of the solution of a linear equation,
which, in view of Theorem~\ref{FLFL}, can be prescribed in an essentially arbitrary way.

In this spirit, if the linearization procedure produced a ``good approximation''
of the nonlinear geometric problem, one would expect that the original nonlocal
minimal graph~$u$ is well approximated near the origin
by a term of the form~$\e\,\bar a \,x_2^{\sigma}$,
with no prescription whatsoever on~$\bar a$. Quite surprisingly, formula~\eqref{VANP}
(applied here with~$\zeta_0:=0$) tells us that this is not the case,
and the correct behavior of a nonlocal minimal graph at the boundary
{\em cannot be simply understood ``by linearization''}.
\medskip

\begin{figure}
    \centering
    \includegraphics[width=8.9cm]{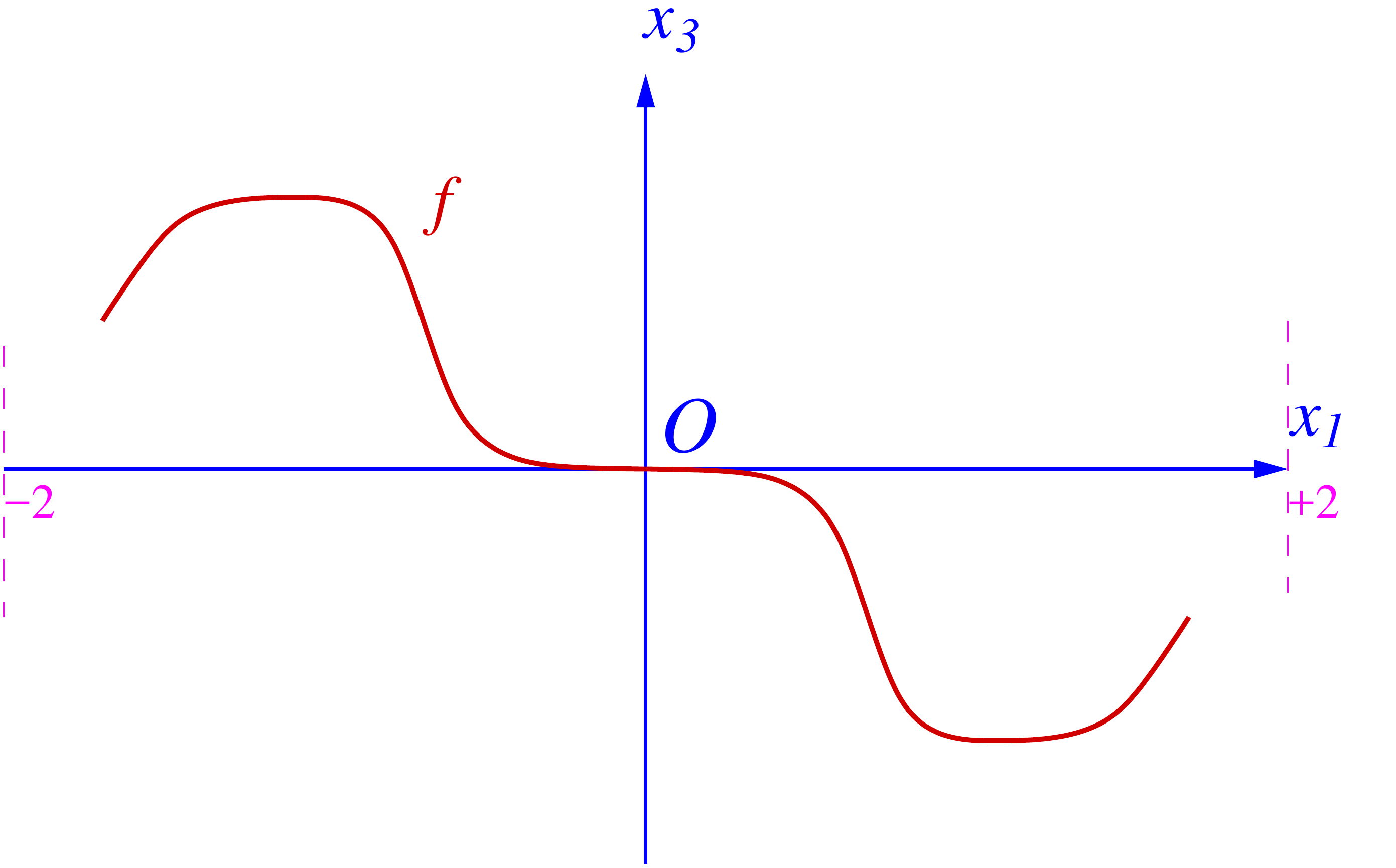}
    \caption{\em {{How we expect the graph of the function~$f$ defined in~\eqref{MR3516886EQ}
    when~$n=2$ and the exterior datum is
    $$u(x):=\chi_{(-2,0)\times(-1,-3)}(x)-
    \chi_{(0,2)\times(-1,-3)}$$
    for all~$x\in\R^2\setminus\big( (-2,2)\times(0,4)\big)$. Notice the horizontal tangency at the origin.}}}
    \label{DOMA-2}
\end{figure}

Theorem~\ref{NLMS} also reveals a {\em structural difference}
of the boundary regularity theory of fractional
minimal surfaces embedded in~$\R^{n+1}$ when~$n=2$
with respect to the case in which~$n=1$.
Indeed, when~$n=1$, an $s$-minimal
graph in which the exterior datum is attained
continuously at a boundary point is necessarily~$C^{1,\frac{1+s}2}$
in a neighborhood of such a point (see Theorem~1.2
in~\cite{2019arXiv190405393D}).
Instead, when~$n=2$, a similar result does not hold:
as a matter of fact, when~$n=2$,
\begin{itemize}\item
Theorem~\ref{NLMS}
guarantees that 
boundary points which attain the flat exterior datum
in a continuous way have necessarily horizontal tangency,\item
conversely, boundary points in which
the $s$-minimal graph experience a jump have necessarily
a vertical tangency (see~\cite{MR3532394}).
\end{itemize}
Consequently, {\em points with vertical tangency
accumulate to zero crossing points possessing horizontal tangency},
preventing a differentiable boundary regularity of the surface
in a neighborhood of the latter type of points.
\medskip

The proof of Theorem~\ref{NLMS} will require a fine understanding of
fractional minimal
homogeneous graphs.
Indeed, as a pivotal step towards the proof of Theorem~\ref{NLMS}, we establish the surprising feature
that {\em if a homogeneous fractional locally minimal graph 
vanishes in~$\{x_n<0\}$ and it is continuous at the origin,
then it necessarily vanishes at all points of~$\{x_n=0\}$}.
The precise result that we obtain is the following:

\begin{theorem}\label{QC}
Let~$u$ be an $s$-minimal graph in~$\R\times(0,+\infty)$.
Assume that
\begin{equation}\label{ADX}
{\mbox{$u(x_1,x_2)=0$ if~$x_2<0$.}}\end{equation}
Assume also that~$u$ is positively homogeneous of degree~$1$, i.e.
\begin{equation}\label{ADX6}
u(tx)=tu(x)\qquad{\mbox{ for all $x\in\R^2$ and~$t>0$.}}
\end{equation}
Suppose that
\begin{equation}\label{ADX3}
u(0):=\lim_{x\to0} u(x)=0.
\end{equation}
Then~$u(x)=0$ for all~$x\in\R^2$.
\end{theorem}

We take this opportunity to state and discuss
{\em some new interesting research lines} opened by the results
obtained in the present paper.

\begin{open}[Vertical tangencies]\label{16}{\rm
In the setting of Theorem~\ref{NLMS}, can one construct examples
in which~$f'(\zeta)=\pm\infty$ for some~$\zeta\in\left(-\frac32,\frac32\right)$?

Namely, is it possible to construct examples of nonlocal
minimal graphs embedded in~$\R^3$ which are flat from outside
and whose trace develops vertical tangencies?

The trace of such possible pathological examples is depicted in Figure~\ref{DOMA-122}.
}\end{open}

\begin{figure}
    \centering
    \includegraphics[width=8.9cm]{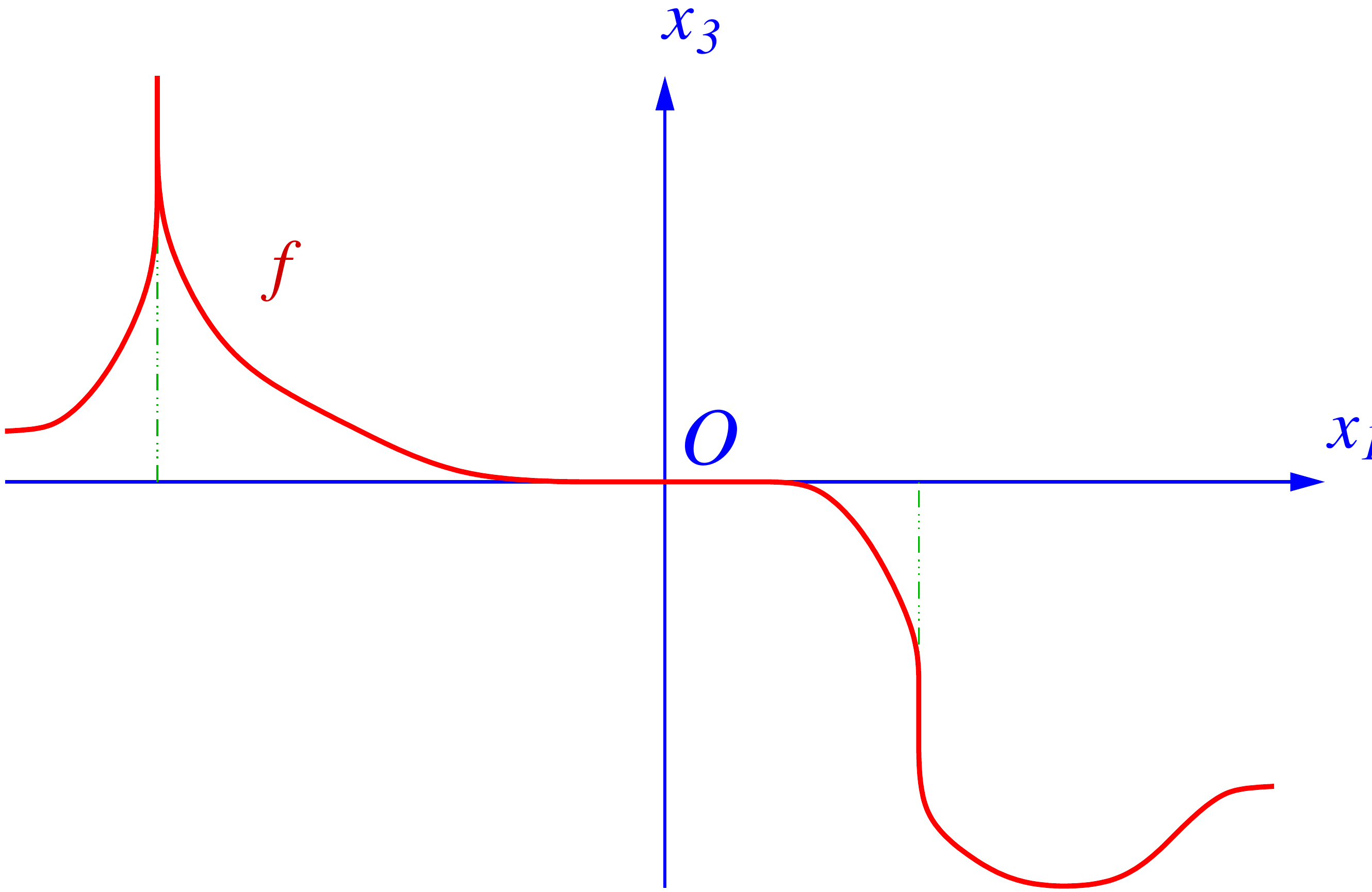}
    \caption{\em {{Open problem~\ref{16}: is it possible to
    construct nonlocal minimal graphs traces
    with vertical tangencies?}}}
    \label{DOMA-122}
\end{figure}

\begin{open}[The higher dimensional case]\label{17}{\rm
It would be interesting to determine whether or not
a result similar to Theorem~\ref{NLMS} holds true in higher
dimension. Similarly, it would be interesting to determine the possible
validity of Theorem~\ref{QC} in higher dimensions.

{F}rom the technical point of view, we observe that some of the auxiliary
results exploited towards the proof of Theorem~\ref{QC} (such as
Lemma~\ref{STEP1} and
Corollary~\ref{STEP11}) are expected to carry over in higher dimension,
therefore one can in principle try to argue by induction, supposing
that a statement such as the one of Theorem~\ref{QC} holds true in dimension~$n$
with the aim of proving it in dimension~$n+1$. The catch in this argument
is that one is led to study the points at which the gradient
of the trace attains its maximal, and this makes an important
connection between this line of research and that of Open Problem~\ref{16}.
}\end{open}

\begin{open}[Behavior at a corner of the domain]{\rm It
would be interesting to detect the behavior of a nonlocal
minimal graph and of its trace at the corners of the domain
and in their vicinity,
in particular understanding (dis)continuity and tangency properties,
possibly also in relation with the convexity or concavity of the corner.
This is related to the analysis of nonlocal minimal cones
in either convex or concave sectors with zero exterior datum.

As a first step towards it, one can try to understand
how to complete
Figure~\ref{DOMA-2} near~$x_1=\pm2$.}\end{open}

\subsection{Organization of the paper} The rest of this manuscript
is organized as follows. Theorem~\ref{FLFL}
is proved in Section~\ref{ISsta}. The arguments used
will exploit a method that we have recently introduced in~\cite{MR3626547}
to show that ``all functions are locally fractional harmonic'', and a careful
discussion of the homogeneous solutions of fractional equations
on cones, see~\cite{MR2075671, MR3810469}.

Then, in Section~\ref{15kdttd3} we present the proof of
Theorem~\ref{QC}. The arguments used here
exploit and develop a series of
fine methods from the theory of nonlocal equations,
comprising boundary Lipschitz bounds, blow-up classification results,
continuity implies differentiability results, nonlocal geometric equations
and nonlocal obstacle-type problems.

In Section~\ref{HDB} we construct a useful barrier,
that we exploit to rule out the case of boundary Lipschitz
singularities for nonlocal minimal graphs.

The proof of Theorem~\ref{NLMS} is contained in Section~\ref{657699767395535}.
Finally, in Appendix~\ref{APPA}, we point out the classical
analogue of Theorem~\ref{FLFL}.

\section{Proof of Theorem~\ref{FLFL}}\label{ISsta}

One important ingredient towards the proof of Theorem~\ref{FLFL}
consists in the construction of a homogeneous solution of a linear fractional
equation with a suitable growth from the vertex of a cone.
This is indeed a classical topic of research, which
also bridges mathematical analysis and probability, see~\cite{MR1936081, MR2075671, MR3810469}
for specific results on fractional harmonic functions on cones. In our setting,
we can reduce to the two-dimensional case (though the higher dimensional case
can be treated in a similar way), and,
for any~$\alpha>0$, we let
\begin{equation}\label{89apq0}
{\mathcal{C}}_\alpha:=\left\{ x=(x_1,x_2)\in\R^{2}
{\mbox{ s.t. $ x_2>0$ and $x_1+\displaystyle\frac{ x_2 }{\alpha }>0
$}}\right\},\end{equation}
and the result that we need is the following one:

\begin{lemma} \label{BAN}
For every~$\sigma\in(0,1)$ and every~$\vartheta\in\R$
there exist~$\alpha>0$, $\beta\in(0,2\sigma)\setminus\{ \vartheta\}$
and~$\bar{v}:S^1\to [0,+\infty)$ with~$\| \bar{v}\|_{L^\infty(S^{1})}=1$
such that the function
\begin{equation}\label{eca-a} \R^2\ni x\mapsto v(x):=|x|^\beta\,\bar{v}\left( \frac{x}{|x|}\right)\end{equation}
satisfies
\begin{equation}\label{eca} \begin{cases}
(-\Delta)^\sigma v = 0 & {\mbox{ in }}{\mathcal{C}}_\alpha,\\
v=0 & {\mbox{ in }}\R^2\setminus{\mathcal{C}}_\alpha.\end{cases}\end{equation}
\end{lemma}

\begin{proof} Let~$\vartheta\in\R$.
By Theorem~3.2 in~\cite{MR2075671},
for every~$\alpha>0$ there exist~$\beta(\alpha)\in(0,2\sigma)$
and~$\bar{v}_\alpha:S^1\to [0,+\infty)$ with~$\| \bar{v}_\alpha\|_{L^\infty(S^{1})}=1$
such that the function
$$ \R^n\ni x\mapsto v_\alpha(x):=|x|^\beta\,\bar{v}_\alpha\left( \frac{x}{|x|}\right)$$
satisfies
$$ \begin{cases}
(-\Delta)^\sigma v_\alpha = 0 & {\mbox{ in }} {\mathcal{C}}_\alpha,\\
v_\alpha=0 & {\mbox{ in }}\R^2\setminus{\mathcal{C}}_\alpha.\end{cases}$$
Let us focus on the case~$\alpha=1$. If~$\beta(1)\ne\vartheta$,
then the claims of Lemma~\ref{BAN} are satisfied
by choosing~$\alpha:=1$, $\beta:=\beta(1)$
and~$\bar{v}:=\bar{v}_1$.

If instead~$\beta(1)=\vartheta$, we exploit Lemma~3.3 in~\cite{MR2075671}.
Namely, since~${\mathcal{C}}_{1/2}\supset {\mathcal{C}}_1$,
we deduce from Lemma~3.3 in~\cite{MR2075671} that~$\beta(1/2)<
\beta(1)=\vartheta$, and thus
the claims of Lemma~\ref{BAN} are satisfied in this case
by choosing~$\alpha:=1/2$, $\beta:=\beta(1/2)$
and~$\bar{v}:=\bar{v}_{1/2}$.
\end{proof}

Exploiting Lemma~\ref{BAN}, we will obtain that the boundary derivatives
of~$\sigma$-harmonic functions have maximal span.
For this, given~$k\in\N$,
we define by~${\mathcal{H}}_k$ the set of all functions~$u\in C(\R^n)$
for which there exists~$r>0$ such that
\begin{equation}\label{KK} \begin{cases}
(-\Delta)^\sigma u=0 & {\mbox{ in }} B_r\cap \{x_n>0\},\\
u=0 & {\mbox{ in }} \{x_n<0\},\end{cases}\end{equation}
and for which there exists~$\phi\in C^k(\overline{B'_r})$ such that
$$ {\mathcal{T}}u(x'):=
\lim_{x_n\searrow0} \frac{u(x',x_n)}{x_n^\sigma}=\phi(x')\qquad{\mbox{ for all }}
x'\in B_r'.$$
Also, given~$r>0$ and~$\phi\in C^k(\overline{B'_r})$, we consider the array
\begin{equation}\label{mul}{\mathcal{D}}^k\phi(0):=\big(
D^\gamma \phi(0)
\big)_{{\gamma\in\N^{n-1}}\atop{|\gamma|\le k}}\;.\end{equation}
Namely, the array~${\mathcal{D}}^k\phi(0)$ contains all the derivatives
of~$\phi$ at the origin, up to order~$k$.
Fixing some order in the components of the multiindex, we can consider~${\mathcal{D}}^k\phi(0)$
as a vector in~$\R^{N_k}$, with
\begin{equation}\label{ennekappa} N_k:=\sum_{j=0}^k (n-1)^j.\end{equation}
The following
result states that the linear space produced in this way is ``as large as possible'':

\begin{lemma}\label{fan}
We have that
\begin{equation}\label{SP} \Big\{ {\mathcal{D}}^k\phi(0), {\mbox{ with }}\phi={\mathcal{T}}u
{\mbox{ and }}u\in {\mathcal{H}}_k\Big\}=\R^{N_k}.\end{equation}
\end{lemma}

\begin{proof} Suppose, by contradiction, that the linear space in the left
hand side of~\eqref{SP} does not exhaust the whole of~$\R^{N_k}$.
Then, there would exist
\begin{equation}\label{0876tg456}
\omega\in \R^{N_k}\setminus\{0\}\end{equation}
such that
the linear space in the left
hand side of~\eqref{SP} lies in the orthogonal space of~$\omega$. Namely,
for every~$u\in {\mathcal{H}}_k$ with~$\phi={\mathcal{T}}u$, 
\begin{equation}\label{mul2}
{\mathcal{D}}^k\phi(0)\cdot\omega=0.
\end{equation}
Recalling the notation in~\eqref{mul}, we can write~$\omega=(\omega_\gamma)_{{\gamma\in\N^{n-1}}\atop{|\gamma|\le k}}$,
and then~\eqref{mul2} takes the form
\begin{equation}\label{mul3}
\sum_{{\gamma\in\N^{n-1}}\atop{|\gamma|\le k}}\omega_\gamma\,
D^\gamma\phi(0)=0,\qquad
{\mbox{ for all~$u\in {\mathcal{H}}_k$ with~$\phi={\mathcal{T}}u$. }}
\end{equation}
Now, we exploit Lemma~\ref{BAN} with~$\vartheta:=\sigma$.
In the notation of Lemma~\ref{BAN}, 
we take~$P=(P_1,P_2)\in{\mathcal{C}}_\alpha\subset\R^2$ such that~$v(P)>0$
and define
$$\tilde v(x_1,x_2):=\frac{v(P)\;x_2^\sigma}{P_2^\sigma}.$$
We observe that~$\tilde v(P)=v(P)$.
Accordingly, by the Boundary Harnack Inequality
(see Theorem~1 on page~44 of~\cite{MR1438304}),
we have that, for all~$\e>0$,
\begin{equation}\label{5302345}
\left[ \frac1C,C\right]\ni \Lambda:=\lim_{x_n\searrow 0} \frac{v(\e,x_n)}{
\tilde v(\e,x_n)}=
\lim_{\tau\searrow 0} \frac{P_2^\sigma \;v(\e,\tau)}{
v(P)\;\tau^\sigma},
\end{equation}
for some~$C\ge1$.

Now, for all~$\e>0$
and~$\zeta\in S^{n-2}$
we define
\begin{equation}\label{eca2} w(x):=v(x'\cdot\zeta+\e,x_n).\end{equation}
Also, if~$r:=\frac{\e}{2}$ and~$x\in B_r\cap\{x_n>0\}$, it follows
that~$x'\cdot\zeta+\e+\frac{x_n}{\alpha}\ge -r+\e>0
$, and thus~$(x'\cdot\zeta+\e,x_n)\in
{\mathcal{C}}_\alpha$. Consequently, by~\eqref{eca},
we have that
\begin{equation}\label{eca4}
{\mbox{$(-\Delta)^\sigma w=0$ in~$B_r\cap\{x_n>0\}$.}}\end{equation}
Moreover, using~\eqref{eca} and~\eqref{eca2}, we see that 
\begin{equation}\label{eca5}
{\mbox{if~$x_n<0$
then~$w(x)=0$. }}\end{equation}
{F}rom~\eqref{eca4} and~\eqref{eca5} (see e.g.
Section~1.1 in~\cite{MR3694738}, or~\cites{MR3293447, MR3276603}),
it also follows that the function~${\rm dist}_{\{x_n<0\}}^{-\sigma} w$
belongs to~$C^\infty(B_r\cap\overline{ \{x_n>0\} })$, hence we can define
\begin{equation}\label{2389IKA}
\psi(x'):=\lim_{x_n\searrow0} \frac{w(x',x_n)}{{\rm dist}_{\{x_n<0\}}^\sigma(x',x_n)}
=
\lim_{x_n\searrow0} \frac{w(x',x_n)}{x_n^\sigma}={\mathcal{T}}w(x').\end{equation}
This, \eqref{eca4} and~\eqref{eca5} give that~$
w\in{\mathcal{H}}_k$.
As a consequence, recalling~\eqref{mul3}, we have that
\begin{equation}\label{PL} \sum_{{\gamma\in\N^{n-1}}\atop{|\gamma|\le k}}
\omega_\gamma\,
D^\gamma\psi(0)=0.\end{equation}
In view of~\eqref{eca2} and~\eqref{2389IKA}, we also observe that
\begin{equation}\label{8343546}
\psi(x')=
\lim_{x_n\searrow0} \frac{v(x'\cdot\zeta+\e,x_n)}{x_n^\sigma}.
\end{equation}
Accordingly, recalling~\eqref{5302345},
\begin{equation}\label{8343546BIS}
\psi(0)=
\lim_{x_n\searrow0} \frac{v(\e,x_n)}{x_n^\sigma}=\frac{\Lambda\,v(P)}{P_2^\sigma}=:\tilde\Lambda\ne0.
\end{equation}
Also, by~\eqref{8343546}, we can write that
\begin{equation}\label{7u12jd2}
\psi(x')=\psi_0(x'\cdot\zeta),\end{equation}
for some~$\psi_0:\R\to[0,+\infty)$.
Hence,
recalling the homogeneity in~\eqref{eca-a},
for all~$t>0$,
\begin{eqnarray*}
&&\psi_0\big(t (x'\cdot\zeta+\e)-\e\big)=
\psi\big(t (x'+\e\zeta)-\e\zeta\big)=
\lim_{x_n\searrow0} \frac{v\big((t (x'+\e\zeta)-\e\zeta)\cdot\zeta+\e, x_n\big)}{x_n^\sigma}\\
&&\qquad=
\lim_{x_n\searrow0} \frac{v(
t x'\cdot\zeta+t\e, x_n)}{x_n^\sigma}=
\lim_{y_n\searrow0} 
\frac{v\big(
t (x'\cdot\zeta+\e), ty_n\big)}{t^\sigma y_n^\sigma}
=
\lim_{y_n\searrow0} \frac{t^\beta v(
x'\cdot\zeta+\e, y_n)}{t^\sigma y_n^\sigma}\\&&\qquad=
t^{\beta-\sigma}\psi(x')=t^{\beta-\sigma}\psi_0(x'\cdot\zeta)
\end{eqnarray*}
Taking $m$ derivatives in~$t$ of this identity, we conclude that
\begin{equation}\label{evrt} 
\psi_0^{(m)}\big(t (x'\cdot\zeta+\e)-\e\big)\,
(x'\cdot\zeta+\e)^m
=\prod_{i=0}^{m-1}(\beta-\sigma-i)\;t^{\beta-\sigma-m}\,\psi_0(x'\cdot\zeta).\end{equation}
In addition, by~\eqref{7u12jd2}, we have that~$D^\gamma\psi(x')=
\zeta^\gamma\psi_0^{(|\gamma|)}(x'\cdot\zeta)$. Hence, evaluating~\eqref{evrt}
at~$t:=1$ and~$x':=0$, if~$m=|\gamma|$ we have that
\begin{equation}\label{5379x34} \begin{split}&\e^m D^\gamma\psi(0)=
\zeta^\gamma\psi_0^{(m)}(0)\,\e^m
=\zeta^\gamma \kappa_m\,\psi_0(0)
=\zeta^\gamma \kappa_m\,\psi(0),\\
{\mbox{where }}\qquad&\kappa_m:=\prod_{i=0}^{m-1}(\beta-\sigma-i)
.\end{split}\end{equation}
Plugging this information into~\eqref{PL}, and recalling~\eqref{8343546BIS},
we find that
\begin{equation}\label{345732}\begin{split}&
0=\frac1{\tilde\Lambda}\sum_{m=1}^k
\sum_{{\gamma\in\N^{n-1}}\atop{|\gamma|= m}}
\omega_\gamma\,
D^\gamma\psi(0)=\frac1{\tilde\Lambda}
\sum_{m=1}^k
\sum_{{\gamma\in\N^{n-1}}\atop{|\gamma|= m}}
\kappa_m\,\e^{-m}
\omega_\gamma \zeta^\gamma\psi(0)\\&\qquad\qquad=\frac1{\tilde\Lambda}
\sum_{{\gamma\in\N^{n-1}}\atop{|\gamma|\le k}}
\kappa_{|\gamma|}\,
\omega_\gamma \left(\frac{\zeta}{\e}\right)^\gamma\psi(0)=
\sum_{{\gamma\in\N^{n-1}}\atop{|\gamma|\le k}}
\kappa_{|\gamma|}\,
\omega_\gamma \left(\frac{\zeta}{\e}\right)^\gamma
.\end{split}\end{equation}
Since~$\frac{\zeta}\e$ ranges in an open set of~$\R^{n-1}$,
we deduce from~\eqref{345732} and the 
Identity Principle for polynomials that~$\kappa_{|\gamma|}\,
\omega_\gamma=0$ for all~$\gamma\in\N^{n-1}$ with~$|\gamma|\le k$.
As a consequence, by~\eqref{5379x34}, we find that~$
\omega_\gamma=0$ for all~$\gamma\in\N^{n-1}$ with~$|\gamma|\le k$,
hence~$\omega=0$. This is in contradiction with~\eqref{0876tg456}
and thus we have proved the desired result.
\end{proof}

With this, we are in the position of completing the proof of Theorem~\ref{FLFL}
by arguing as follows: 

\begin{proof}[Proof of Theorem~\ref{FLFL}]\label{89-ppA}
Since the claims in Theorem~\ref{FLFL} have a linear structure
in~$f$, $f_\e$ and~$u_\e$, by the Stone-Weierstra{\ss} Theorem,
it is enough to prove Theorem~\ref{FLFL} if~$f$ is a monomial.
Hence, we fix~$\e\in(0,1)$, possibly to be taken conveniently small,
and we suppose that
\begin{equation}\label{mu1}
f(x')=\frac{(x')^\mu}{\mu!}\qquad{\mbox{for some }}\mu\in\N^{n-1}.\end{equation}
Then, we apply Lemma~\ref{fan}, finding a suitable function~$u_\star
\in{\mathcal{H}}_k$ with~$\phi_\star:={\mathcal{T}}u_\star$ that satisfies
\begin{equation}\label{mu2} D^\gamma \phi_\star(0)=\begin{cases}
1 & {\mbox{ if }} \gamma=\mu,\\
0 & {\mbox{ if $|\gamma|\le |\mu|+k$ and~$ \gamma\ne\mu$.}}
\end{cases}\end{equation}
We define
\begin{eqnarray*}&&
u_\e(x):= \e^{-\sigma-|\mu|} \,u_\star(\e x)\\
{\mbox{and }}&& f_\e(x'):=\e^{-|\mu|} \,\phi_\star(\e x').
\end{eqnarray*}
Then, if~$x\in B_1\cap\{x_n>0\}$, we have that~$(-\Delta)^\sigma u_\e(x)=
\e^{\sigma-|\mu|} \,(-\Delta)^\sigma u_\star(\e x)=0$
as long as~$\e$ is sufficiently small. In addition,
we see that~$u_\e<0$ in~$\{x_n<0\}$, and that
$$ \lim_{x_n\searrow0} \frac{u_\e(x',x_n)}{x_n^\sigma}=
\lim_{x_n\searrow0} \frac{\e^{-\sigma-|\mu|} \,u_\star(\e x',\e x_n)}{x_n^\sigma}=
\lim_{\tau\searrow0} \frac{\e^{-|\mu|} \,u_\star(\e x',\tau)}{\tau^\sigma}=\e^{-|\mu|} \,\phi_\star(\e x')=
f_\e(x').
$$
These observations prove~\eqref{56:01} and~\eqref{56:02}.

Furthermore, 
$$ D^\gamma f_\e(x')=\e^{|\gamma|-|\mu|} \,D^\gamma \phi_\star(\e x').$$
Consequently, if we set~$g_\e(x'):=f_\e(x')-f(x')$, we deduce
from~\eqref{mu1} and~\eqref{mu2} that
$$ D^\gamma g_\e(0)=0,$$
for all~$\gamma\in\N^{n-1}$ such that~$|\gamma|\le|\mu|+k$.

This observation and a Taylor expansion give that, for all~$x'\in B_1'$
and all~$\zeta
\in\N^{n-1}$ such that~$|\zeta|\le k$,
\begin{eqnarray*}&&
| D^\zeta f_\e(x') - D^\zeta f(x')|=
|D^\zeta g_\e(x')|\le C_k\,\sup_{y'\in B_1'}
\sum_{{\alpha\in\N^n}\atop{|\alpha| =|\mu|+k+1}}
|D^\alpha g_\e(y')|\\&&\qquad=
C_k\,\sup_{y'\in B_1'}
\sum_{{\alpha\in\N^n}\atop{|\alpha| =|\mu|+k+1}}
|D^\alpha f_\e(y')|=
C_k\,\sup_{y'\in B_1'}
\sum_{{\alpha\in\N^n}\atop{|\alpha| =|\mu|+k+1}}
\e^{|\alpha|-|\mu|} \,
|D^\alpha \phi_\star(\e x')|\\&&\qquad
= C_k\,\e^{k+1} \,\sup_{y'\in B_1'}
\sum_{{\alpha\in\N^n}\atop{|\alpha| =|\mu|+k+1}}
|D^\alpha \phi_\star(\e x')|\le C_k'\,\e,
\end{eqnarray*}
for some~$C_k$, $C_k'>0$. This establishes~\eqref{56:03}, up to renaming~$\e$.
\end{proof}

\section{Proof of Theorem~\ref{QC}}\label{15kdttd3}

In this section, for the sake of generality,
some results are proved in arbitrary dimension~$n\ge2$,
whenever the proof would not experience significant
simplifications in the case~$n=2$ (then, for the proof of
Theorem~\ref{QC}, we restrict ourselves to the case~$n=2$,
see also Open Problem~\ref{17}).
As customary, given~$E\subset\R^{n+1}$,
it is convenient to consider the nonlocal mean curvature
at a point~$x\in\partial E$, defined by
\begin{equation}\label{jfgnjbj96768769} {\mathcal{H}}_E^s(x):=\int_{\R^{n+1}}
\frac{\chi_{\R^{n+1}\setminus E}(y)-\chi_E(y)}{|x-y|^{n+1+s}}\,dy.\end{equation}
The first step to prove Theorem~\ref{QC} is to establish
the existence of a small vertical cone not intersecting
the boundary of a homogeneous nonlocal minimal surface
on a hyperplane with null exterior datum. Letting~$e_{n+1}:=(0,\dots,0,1)\in\R^{n+1}$,
the precise result that we have
is the following one:

\begin{lemma}\label{3243-237gndlsm}
Let~$E\subset\R^{n+1}$ be a locally $s$-minimal set in~$\R^n\times(0,+\infty)$.
Assume that
\begin{equation}\label{ADX-EE}
E\cap\{x_n<0\}=\{x_{n+1}<0\}\cap\{x_n<0\}\end{equation}
and that
\begin{equation}\label{ADX6-EE}
{\mbox{$tE=E$ for every~$t>0$.}}
\end{equation}
Then,
\begin{equation}\label{6BOu-1BIS}
\min\Big\{ {\rm dist}\big(e_{n+1}, (\partial E)\cap\{x_n>0\}\big),\;
{\rm dist}\big(-e_{n+1},(\partial E)\cap\{x_n>0\}\big)\Big\}>0.
\end{equation}
\end{lemma}

\begin{proof}
We claim that
\begin{equation}\label{6BOu-1} e_{n+1}\not\in
\overline{(\partial E)\cap\{x_n>0\}}.
\end{equation}
Indeed, suppose by contradiction that~$e_{n+1}\in\overline{(
\partial E)\cap\{x_n>0\}}$,
and thus, by~\eqref{ADX6-EE}, also~$te_{n+1}\in\overline{(\partial E)
\cap\{x_n>0\}}$,
for all~$t>0$.
By~\eqref{ADX-EE}, we have that~$B_{1/2}(e_{n+1})\cap\{x_n<0\}\subset \R^{n+1}\setminus E$.
Hence (see e.g. Theorem~B.9 in~\cite{MR3926519}), we have that~${\mathcal{H}}_{E}^s(te_{n+1})=0$
for all~$t\in\left[\frac{9}{10},\frac{11}{10}\right]$.
In particular, if~$F:=E-\frac{e_{n+1}}{100}$,
we see that~$F\subset E$, and thus
\begin{eqnarray*}&& 0={\mathcal{H}}_{E}^s(e_{n+1})
-{\mathcal{H}}_{F}^s(e_{n+1})
=-2\int_{\R^{n+1}}
\frac{\chi_{E\setminus F}(y)}{|e_{n+1}-y|^{n+1+s}}\,dy\\&&\qquad\le-2
\int_{\R^{n-1}\times(-\infty,0)\times\left(-\frac1{100},0\right)}
\frac{dy}{|e_{n+1}-y|^{n+1+s}}<0.\end{eqnarray*}
This contradiction proves~\eqref{6BOu-1}.

Similarly, one proves that~$-e_{n+1}\not\in
\overline{(\partial E)\cap\{x_n>0\}}$.
Consequently, since~$\partial E$ is a closed set, we obtain~\eqref{6BOu-1BIS},
as desired.
\end{proof}

As a byproduct of Lemma~\ref{3243-237gndlsm}, we obtain
that the second blow-up of an $s$-minimal graph which is flat
from one side is necessarily a graph as well
(see e.g. Lemmata~2.2 and~2.3
in~\cite{2019arXiv190405393D} for the basic properties
of the second blow-up).
The precise result goes as follows:

\begin{lemma}\label{BOAmdfiUAMP}
Let~$u$ be an $s$-minimal graph in~$(-2,2)^{n-1}\times(0,4)$.
Assume that there exists~$h>0$
such that
\begin{equation}\label{568-029-3984112}
{\mbox{$u=0$ in~$ (-2,2)^{n-1}\times(-h,0)$.}}\end{equation}
Let~$E_{00}$ be the second blow-up of~$E_u$, being~$E_u$ defined in~\eqref{LEYUDEF}.
Then:
\begin{eqnarray}\label{POS-E001}
&& {\mbox{either $E_{00}\cap\{x_n>0\}=\varnothing$,}}
\\ \label{POS-E002}&& {\mbox{or $E_{00}\cap\{x_n>0\}=\{x_n>0\}$,}}\\&&
{\mbox{or $E_{00}$
has a graphical structure, namely}}\nonumber\\&&\label{POS-E003}
{\mbox{there exists~$u_{00}:\R^n\to\R$ such that~$E_{00}=E_{u_{00}}$}.}\end{eqnarray}
Furthermore, if~\eqref{POS-E003} holds true, then
\begin{equation}\label{POS-E004}
\lim_{x\to0}u_{00}(x)=0.
\end{equation}
\end{lemma}

\begin{proof} 
We suppose that
\begin{equation}\label{454562382023}
{\mbox{\eqref{POS-E001} and~\eqref{POS-E002} do not hold,}}\end{equation}
and we aim at showing that~\eqref{POS-E003} and~\eqref{POS-E004}
are satisfied. We start by proving~\eqref{POS-E003}.
For this, we first observe that~$E_{00}$
has a generalized
``hypographical'' structure,
that is
\begin{equation}\label{657865-349}
{\mbox{if~$(y,y_{n+1})\in(\R^n\times\R)\cap E_{00}$,
then~$(y,\tau)\in E_{00}$ for every~$\tau\le y_{n+1}$.}}\end{equation}
Indeed, each rescaling of~$E_u$ has such property,
and since these rescalings approach~$E_{00}$
in the Hausdorff distance (see~\cite{2019arXiv190405393D}),
the claim in~\eqref{657865-349} follows.

Moreover, by~\eqref{568-029-3984112},
\begin{equation}\label{ADX-EE-568-029-3984112}
E_{00}\cap\{x_n<0\}=\{x_{n+1}<0\}\cap\{x_n<0\},\end{equation}
and, by Lemma 2.2 in~\cite{2019arXiv190405393D},
we have that
\begin{equation}\label{ADX-EE-568-029-3984112-NIA}
{\mbox{$tE_{00}=E_{00}$ for every~$t>0$.
}}\end{equation}
{F}rom this and~\eqref{ADX-EE-568-029-3984112}, we are
in the position of using
Lemma~\ref{3243-237gndlsm},
and thus deduce from~\eqref{6BOu-1BIS} that
\begin{equation}\label{6BOu-1BIS-TRIS}
\min\Big\{ {\rm dist}\big(e_{n+1},(\partial E_{00})\cap\{x_n>0\}\big),\;
{\rm dist}\big(-e_{n+1},(\partial E_{00})\cap\{x_n>0\}\big)\Big\}>0.\end{equation}
Then, for all~$x\in\R^{n-1}\times(0,+\infty)$, we set
$$ u_{00}(x):=\sup\big\{
y {\mbox{ s.t. }}(x,y)\in E_{00}
\big\}.$$
By extending~$u_{00}$ to vanish in~$\R^{n-1}\times(-\infty,0)$,
we find that~$E_{00}$ is the subgraph of~$u_{00}$,
as desired. To this aim, it remains to prove that the image of~$u_{00}$
is~$\R$, namely that for every~$x\in\R^{n-1}\times(0,+\infty)$,
\begin{equation}\label{71-02-23884oe}
u_{00}(x)\not\in\{-\infty,+\infty\}.
\end{equation}
As a matter of fact, in light of~\eqref{ADX-EE-568-029-3984112-NIA},
it is sufficient to prove~\eqref{71-02-23884oe}
for every~$x\in S^{n-1}_+:=\{x=(x_1,\dots,x_n)\in\R^n
{\mbox{ s.t. $|x|=1$ and~$x_n>0$}}\}$.
Hence, we set
\begin{eqnarray*}&&
\omega :=\big\{ x\in S^{n-1}_+{\mbox{ s.t. }}u_{00}(x)\not\in\{-\infty,+\infty\}\big\},\\
&&
\omega_+ :=\big\{ x\in S^{n-1}_+{\mbox{ s.t. }}u_{00}(x)=+\infty\big\}\\
{\mbox{and }}&&
\omega_-:=\big\{ x\in S^{n-1}_+{\mbox{ s.t. }}u_{00}(x)=-\infty\big\},
\end{eqnarray*}
and, to prove~\eqref{POS-E003},
we want to show that~$\omega_+=\varnothing=\omega_-$.

For a contradiction, assume that~$\omega_-\ne\varnothing$.
By~\eqref{454562382023}, we also know that~$\omega_-\ne S^{n-1}_+$.
Hence, we can take~$x_-\in\omega_-$
and~$x_+\in\omega\cup\omega_+$.
By construction, we have that~$u_{00}(x_-)=-\infty$
and~$u_{00}(x_+)\in\R\cup\{+\infty\}$, and therefore~$(x_-,y)\in
\R^{n+1}\setminus E_{00}$ for all~$y\in\R$, while~$(x_+,y_+)\in E_{00}$
for some~$y_+\in\R$.

This and~\eqref{657865-349} give that~$(x_+,y)\in E_{00}$
for all~$y\le y_+$.
In particular, for all~$k\in\N$
sufficiently large, we have that~$(x_+,-k)\in E_{00}$,
with~$(x_-,-k)\in \R^{n+1}\setminus E_{00}$.
Consequently, by~\eqref{ADX-EE-568-029-3984112-NIA}, we see that~$
\left(\frac{x_+}k,-1\right)\in E_{00}$
and~$\left(\frac{x_-}k,-1\right)\in \R^{n+1}\setminus E_{00}$.
As a result, by taking the limit as~$k\to+\infty$, we conclude that~$-e_{n+1}\in
\overline{(\partial E_{00})\cap\{x_n>0\}}$. But this is in contradiction
with~\eqref{6BOu-1BIS-TRIS}, and therefore necessarily~$\omega_-=\varnothing$.

Similarly, one proves that~$\omega_+=\varnothing$,
and this completes the proof of~\eqref{71-02-23884oe}.

The proof of~\eqref{POS-E003} is thus completed, and we now
focus on the proof of~\eqref{POS-E004}.
For this, assume the converse: then, there exists
a sequence~$x^{(k)}=(x^{(k)}_1,\dots,x^{(k)}_n)$ such that~$x^{(k)}_n>0$,
$x^{(k)}\to0$ as~$k\to+\infty$ and~$|u(x^{(k)})|\ge a_0$,
for some~$a_0>0$. Up to a sign change, we can suppose that~$
u(x^{(k)})\ge a_0$. Hence, we have that~$( x^{(k)}, u(x^{(k)}))\in
(\partial E_{00})\cap\{x_n>0\}$ and then,
recalling~\eqref{ADX-EE-568-029-3984112-NIA},
we deduce that~$
\left(\frac{ x^{(k)}}{ u(x^{(k)})},1\right)\in
(\partial E_{00})\cap\{x_n>0\}$.

Accordingly, taking the limit as~$k\to+\infty$,
we find that~$e_{n+1}\in\overline{
(\partial E_{00})\cap\{x_n>0\} }$. This is in contradiction
with~\eqref{6BOu-1BIS-TRIS} and, as a result,
the proof of~\eqref{POS-E004} is complete.
\end{proof}

{F}rom Lemma~\ref{3243-237gndlsm} we also deduce
a regularity result of the following type:

\begin{lemma}\label{lemma33}
Let~$u$ be an $s$-minimal graph in~$\R^{n-1}\times(0,+\infty)$.
Assume that
\begin{equation}\label{pppADX}
{\mbox{$u(x_1,\dots,x_n)=0$ if~$x_n<0$.}}\end{equation}
Assume also that~$u$ is positively homogeneous of degree~$1$, i.e.
\begin{equation}\label{pppADX6}
u(tx)=tu(x)\qquad{\mbox{ for all $x\in\R^n$ and~$t>0$.}}
\end{equation}
Then, we have that
\begin{equation}\label{ADX0}
u\in 
L^\infty_{\rm loc}(\R^n)\cap C([-1,1]^{n-1}\times[0,1])\cap
C^\infty( \R^{n-1}\times(0,+\infty)).\end{equation} 
\end{lemma}

\begin{proof} We claim that
\begin{equation}\label{6BOu} u\in
L^\infty_{\rm loc}(\R^n).
\end{equation}
Suppose not. Then there exist~$R>0$ and~$x^{(k)}=(x^{(k)}_1,\dots,x^{(k)}_n)\in B_R$ such that~$|u(x^{(k)})|\ge k$.
Without loss of generality, we can suppose that~$u(x^{(k)})\ge k$.
Then, by~\eqref{pppADX},
we see that~$x^{(k)}_n>0$, and,
recalling~\eqref{pppADX6}, we have that~$\big( t x^{(k)},t u(x^{(k)})\big)\in
\partial E_u$ for every~$t>0$ (being~$E_u$ defined in~\eqref{LEYUDEF}), and then, in particular,
$$ \left( \frac{x^{(k)}}{u(x^{(k)})},1\right)\in\partial E_u.$$
Accordingly, taking the limit as~$k\to+\infty$, 
we find that~$e_{n+1}\in\overline{\partial E_u\cap\{x_n>0\}}$.
This is in contradiction with~\eqref{6BOu-1BIS},
whence the proof of~\eqref{6BOu} is complete. 

As a result, by Theorem~1.1 in~\cite{MR3934589}
we obtain that~$u$ is smooth in~$\{x_n>0\}$,
and the continuity up to~$\{x_n=0\}$ follows from Theorem~1.1
of~\cite{MR3516886}. This proves the claim in~\eqref{ADX0}.
\end{proof}

Now, we show that if the second blow-up is either empty or full
in a halfspace, then the original $s$-minimal graph is necessarily
boundary discontinuous:

\begin{lemma}\label{554327h:AJSJD2374859ty}
Let~$u$ be an $s$-minimal graph in~$(-2,2)\times(0,4)$.
Assume that there exists~$h>0$
such that
\begin{equation}\label{Rg68231opp}
{\mbox{$u=0$ in~$ (-2,2)\times(-h,0)$.}}\end{equation}
Let~$E_{00}$ be the second blow-up of~$E_u$. Then,
\begin{eqnarray}
\label{FAG:1}
&& {\mbox{if~$E_{00}=\R\times(-\infty,0)\times(-\infty,0)$, then }}
\lim_{x_2\searrow0} u(0,x_2)<0;
\\
\label{FAG:2}\nonumber
&& {\mbox{if~$E_{00}=\big(\R\times(-\infty,0)\times(-\infty,0)\big)
\cup \big(\R\times(0,+\infty)\times\R\big)$,}}\\
&&{\mbox{then }}\lim_{x_2\searrow0} u(0,x_2)>0.
\end{eqnarray}
\end{lemma}

\begin{proof} We focus on the proof of~\eqref{FAG:1},
since the proof of~\eqref{FAG:2} is similar. 
We recall
(see~\cite{2019arXiv190405393D})
that, as~$k\to+\infty$,
\begin{equation}\label{INCLA0-QUI}
\begin{split}&{\mbox{up to a subsequence, $\chi_{k E_u}$ converges to~$\chi_{{E_{00}}}$
in~$L^1_{\rm loc}(\R^3)$,}}\\
&{\mbox{and $kE_u$ converges to~$E_{00}$
locally in the Hausdorff distance.}}\end{split}
\end{equation}
We claim that for every~$M>0$ there exists~$k_M\in\N$ such that
if~$k\ge k_M$ then
\begin{equation}\label{INCLA1}(
kE_u)\cap B_M\subseteq\left\{x_2<\frac1M\right\}.
\end{equation}
To check this, we argue for a contradiction and suppose that,
for some~$M>1$,
there are infinitely many~$k$'s for which there exists~$p^{(k)}=(p^{(k)}_1,p^{(k)}_2,p^{(k)}_3)\in
(kE_u)\cap B_M$ with~$p^{(k)}_2\ge\frac1M$.
We observe that~$B_{1/(2M)}(p^{(k)})$ cannot be contained in~$kE_u$, otherwise,
recalling the structure of~$E_{00}$ in~\eqref{FAG:1},
\begin{eqnarray*}&& \int_{ B_{1/(2M)}(p^{(k)}) }|\chi_{k E_u}(x)-\chi_{{{E_{00}}}}(x)|\,dx=
\int_{ B_{1/(2M)}(p^{(k)}) }|\chi_{k E_u}(x)|\,dx\\&&\qquad \qquad = |B_{1/(2M)}(p^{(k)})|=
|B_{1/(2M)}|,\end{eqnarray*}
which is in contradiction with~\eqref{INCLA0-QUI}.

As a result, there exists~$q^{(k)}\in B_{1/(2M)}( p^{(k)} )$
with~$q^{(k)}\in\partial (k E_u)$. In particular, we have that~$q^{(k)}_2\ge
p^{(k)}_2-\frac1{2M}\ge\frac1{2M}$, whence, using the clean ball
condition in~\cite{MR2675483}, there exist~$c\in(0,1)$, $r_0\in\left(0,\frac1{4M}\right)$
and~$\tilde q^{(k)}\in B_{r_0}( q^{(k)})$ such that~$B_{cr_0}( \tilde q^{(k)})\subseteq
(kE_u)\cap B_{r_0}( q^{(k)})$.

We remark that if~$x\in B_{cr_0}( \tilde q^{(k)})$ then
$$ x_2\ge \tilde q^{(k)}_2-cr_0\ge q^{(k)}_2 -| q^{(k)}-\tilde q^{(k)}|-r_0
\ge \frac1{2M}-2r_0>0.$$
Consequently, recalling the structure of~$E_{00}$ in~\eqref{FAG:1},
$$ \int_{ B_{cr_0}( \tilde q^{(k)})  }
|\chi_{k E_v}(x)-\chi_{{{F}}}(x)|\,dx=
\big| B_{cr_0}( \tilde q^{(k)}) \big|=| B_{cr_0}|.
$$
This is in contradiction with~\eqref{INCLA0-QUI}
and so it proves~\eqref{INCLA1}.

Now, for all~$\lambda$, $\tau\in(-2,2)$ and~$t>1$
consider the ball~$B_1(\lambda,t,\tau)$. By~\eqref{INCLA1}, there exists~$k_0\in\N$ such that
if~$t\ge 2$ we have that
\begin{equation}\label{980-029-1}
B_1(\lambda,t,\tau)\subseteq \R^3\setminus (kE_u),
\end{equation}
for all~$\lambda$, $\tau\in(-2,2)$ and~$k\ge k_0$.

Now we claim that the claim in~\eqref{980-029-1}
holds true for all~$t>1$ (and not just~$t\ge2$)
with respect to the same~$k_0$: namely, we show that
for all~$\lambda$, $\tau\in(-2,2)$, $k\ge k_0$
and~$t>1$, we have that
\begin{equation}\label{980-029-2}
B_1(\lambda,t,\tau)\cap\big(\partial (kE_u)\big)=\varnothing.
\end{equation}
Indeed, if not, there would exist~$\lambda$, $\tau\in(-2,2)$
and~$k\ge k_0$, and a suitable~$t_\star>1$, for which~$B_1(\lambda,t_\star,\tau)\cap\big(\partial (kE_u)\big)\not=\varnothing$.
More precisely, if~\eqref{980-029-2} were false,
by~\eqref{980-029-1},
we can
slide~$B_1(\lambda,\cdot,\tau)$ with respect to the parameter~$t$
from the right till it touches~$\partial (kE_u)$, say at a point~$Z=(Z_1,Z_2,Z_3)$.
In this way, find that~$t_\star\in(1,2)$
and~$B_1(\lambda,t,\tau)\subseteq \R^3\setminus (kE_u)$
for all~$t>t_\star$, with~$Z\in(\partial (kE_u))\cap B_1(\lambda,t_\star,\tau)$.

Now, we show that, for any fixed~$M>1$,
\begin{equation}\label{Diar34t5y6be9er}
(kE_u) \cap\{ x_2-Z_2>2\}\subseteq \R^3\setminus B_{M/4}(Z),
\end{equation}
as long as~$k$ is sufficiently large.
Indeed, if not, take~$P=(P_1,P_2,P_3)\in (kE_u)\cap B_{M/4}(Z)$
with~$P_2>Z_2+2$. By construction,
$$ |P_1|\le |P_1-Z_1|+|Z_1-\lambda |+|\lambda |\le
\frac{M}{4}+1+2<\frac{M}{3},$$
as long as~$M$ is sufficiently large, and similarly~$|P_2|<\frac{M}{3}$
and~$|P_3|<\frac{M}{3}$. As a consequence, we have that~$P\in B_M$.
This and~\eqref{INCLA1} give that~$P_2<\frac1M$, and then
$$ \frac1M>P_2>Z_2+2\ge t_\star -|t_\star-Z_2|+2\ge t_\star-1+2\ge 2.$$
This is a contradiction and therefore 
the proof of~\eqref{Diar34t5y6be9er} is complete.

Furthermore, when~$k>2/h$, we observe that
\begin{equation}\label{Diar34t5y6be9er-BIS}\begin{split}&
(kE_u) \cap\{ x_2-Z_2\in(-hk,-2)\}\subseteq\left\{
x_3<ku\left( \frac{x_1}k,\frac{x_2}k\right)\right\}\cap\{ x_2\in(-hk,0)\}\\&\qquad\qquad
\subseteq \{ x_3<0\},\end{split}
\end{equation}
thanks to~\eqref{Rg68231opp}.

Then, as a consequence of~\eqref{Diar34t5y6be9er} and~\eqref{Diar34t5y6be9er-BIS},
\begin{equation*}
\int_{ B_{M/4}(Z)\cap \{2< |x_2-Z_2|<hk\}}\frac{\chi_{\R^3\setminus(kE_u)}(y)
-\chi_{ kE_u }(y)}{|Z-y|^{3+s}}\,dy\ge c,
\end{equation*}
for some universal constant~$c>0$ depending only on~$s$.

On this account, if~$M$ and~$k$ are sufficiently large, we find that
\begin{equation}\label{Nysjdf33-1}
\int_{  \{ |x_2-Z_2|>2\}}\frac{\chi_{\R^3\setminus(kE_u)}(y)
-\chi_{ kE_u }(y)}{|Z-y|^{3+s}}\,dy\ge \frac{c}{2}.
\end{equation}
Moreover, for a sufficiently small~$\rho>0$, using the fact that~$
B_1(\lambda,t_\star,\tau)\subseteq \R^3\setminus (kE_u)$,
we see that
\begin{equation}\label{Nysjdf33-2}
\int_{  B_\rho(Z)}\frac{\chi_{\R^3\setminus(kE_u)}(y)
-\chi_{ kE_u }(y)}{|Z-y|^{3+s}}\,dy\ge -C\rho^{1-s}\ge-\frac{c}{4},
\end{equation}
see e.g. Lemma~3.1 in~\cite{MR3516886} for computational details.

Fixing such a~$\rho$ from now on, we deduce from~\eqref{Nysjdf33-1}
and~\eqref{Nysjdf33-2} that
\begin{equation}\label{Nysjdf33-3}
\int_{ B_\rho(Z)\cup \{ |x_2-Z_2|>2\}}\frac{\chi_{\R^3\setminus(kE_u)}(y)
-\chi_{ kE_u }(y)}{|Z-y|^{3+s}}\,dy\ge \frac{c}{4}.
\end{equation}
Also, by~\eqref{Rg68231opp} and~\eqref{INCLA1},
$$ \int_{\{ |x_2-Z_2|\le 2\}\setminus B_\rho(Z)}\frac{\chi_{\R^3\setminus(kE_u)}(y)
-\chi_{ kE_u }(y)}{|Z-y|^{3+s}}\,dy\ge- \frac{c}{8},$$
as long as~$M$ is sufficiently large. This and~\eqref{Nysjdf33-3} yield that
$$ \int_{\R^3}\frac{\chi_{\R^3\setminus(kE_u)}(y)
-\chi_{ kE_u }(y)}{|Z-y|^{3+s}}\,dy\ge\frac{c}{8}>0,$$
which is a contradiction with the minimality of~$kE_u$.
This completes the proof of~\eqref{980-029-2}.

Consequently, by~\eqref{980-029-2},
$$(-2,2)\times(0,2)\times(-2,2)\subseteq \R^3\setminus (k_0 E_u).$$
Therefore
$$\left(-\frac2{k_0},\frac2{k_0}\right)\times
\left(0,\frac2{k_0}\right)
\times\left(-\frac2{k_0},\frac2{k_0}\right)\subseteq \R^3\setminus  E_u,$$
and thus, for all~$y=(y_1,y_2,y_3)\in\left(-\frac2{k_0},\frac2{k_0}\right)\times
\left(0,\frac2{k_0}\right)
\times\left(-\frac2{k_0},\frac2{k_0}\right)$, we have that~$y_3\ge
u(y_1,y_2)$.
In particular, choosing~$y_1:=0$ and~$y_3:=-\frac1{k_0}$,
we find that, for all~$y_2\in\left(0,\frac2{k_0}\right)$,
$$ u(0,y_2)\le -\frac1{k_0}.$$
Then, we can send~$y_2\searrow0$ and obtain
$$ \limsup_{y_2\searrow0} u(0,y_2)<0.$$
This limit in fact exists, thanks to Theorem~1.1
of~\cite{MR3516886}, whence~\eqref{FAG:1}
follows, as desired.
\end{proof}

Next result discusses how the boundary continuity of a nonlocal minimal
graph at some boundary
points implies the differentiability up to the boundary:

\begin{lemma}\label{STEP1}
Let
the assumptions of Theorem~\ref{QC} hold.
Assume also that
\begin{equation}\label{6tsoe}
\lim_{x\to (1,0)} u(x)=0.
\end{equation}
Then~$u\in C^1(\R^2\setminus\ell_-)$, where~$\ell_-:=(-\infty,0]\times \{0\}$.

Similarly, if
\begin{equation}\label{6tsoe-2}
\lim_{x\to (-1,0)} u(x)=0.
\end{equation}
Then~$u\in C^1(\R^2\setminus\ell_+)$, where~$\ell_+:=[0,+\infty)\times \{0\}$.

If both~\eqref{6tsoe} and~\eqref{6tsoe-2} are satisfied, then~$u\in C^1(\R^2)$.
\end{lemma}

\begin{proof} We suppose that~\eqref{6tsoe}
is satisfied,
since the case in which~\eqref{6tsoe-2} holds true is similar
(and so is the case
in which both~\eqref{6tsoe}
and~\eqref{6tsoe-2} are fulfilled).
Using~\eqref{ADX6} and~\eqref{6tsoe},
for any~$\tau>0$,
\begin{equation}\label{8hsfr}u(\tau,0):=
\lim_{x\to (\tau,0)} u(x)
=0.
\end{equation}
We define
\begin{equation}\label{sucjp}
v(x_1,x_2):=u(x_1+1,x_2).\end{equation}
We consider the homogeneous second blow up~${F}$
of~$E_v$ (see Lemmata~2.2 and~2.3 in~\cite{2019arXiv190405393D}),
and we have that, as~$k\to+\infty$,
\begin{equation}\label{INCLA0}
\begin{split}&{\mbox{up to a subsequence, $\chi_{k E_v}$ converges to~$\chi_{{F}}$
in~$L^1_{\rm loc}(\R^3)$,}}\\
&{\mbox{and $kE_v$ converges to~$F$
locally in the Hausdorff distance.}}\end{split}
\end{equation}
By~\eqref{ADX},
\begin{equation}\label{Gina}
{{F}}\cap\{x_2<0\}=\R\times(-\infty,0)\times(-\infty,0).
\end{equation}
In addition, we claim that
\begin{equation}\label{DUIA} {{F}}+(r,0,0)={{F}},\end{equation}
for all~$r\in\R$.
Indeed, we set~$v_k(x):=kv\left(\frac{x}k\right)$, and,
by~\eqref{ADX6}, we see that, for every~$x\in\R^2$,
\begin{eqnarray*}&&v_k\left(\frac{k+1}{k}x_1+1,\frac{k+1}{k}x_2
\right)=
k v\left( \frac{k+1}{k^2}x_1+\frac1k,\frac{k+1}{k^2}x_2\right)\\&&\quad=
k u\left( \frac{k+1}{k^2}x_1+\frac1k+1,\frac{k+1}{k^2}x_2\right)=
k u\left( \frac{k+1}{k^2}x_1+\frac{k+1}k,\frac{k+1}{k^2}x_2\right)\\&&\quad=
(k+1)\,
u\left( \frac{x_1}k+1,\frac{x_2}k\right)=(k+1)\,
v\left( \frac{x_1}k,\frac{x_2}k\right)=\frac{k+1}{k}\,v_k(x).
\end{eqnarray*}
For this reason, taking the limit as~$k\to+\infty$
and recalling~\eqref{INCLA0},
we see that~$ {{F}}+(1,0,0)={{F}}$.
Since~$F$ is a cone (see Lemma~2.2 in~\cite{2019arXiv190405393D}),
this completes the proof of~\eqref{DUIA}.

Accordingly, by~\eqref{DUIA} and the dimensional reduction
(see~\cite{MR2675483}),
we can write that~${{F}}=\R\times G$, for some cone~$G\subseteq\R^2$
which is locally $s$-minimal in~$(0,+\infty)\times\R$.
In view of~\eqref{Gina} we have that~$G\cap ((-\infty,0)\times\R)=
(-\infty,0)\times(-\infty,0)$. Therefore,
by minimality,
\begin{equation}\label{nosn-21}\begin{split}
&{\mbox{either~$G=\R\times(-\infty,0)$,
or~$G=(-\infty,0)\times(-\infty,0)$,}}\\&
{\mbox{or~$G=\big( (-\infty,0)\times(-\infty,0)\big)\cup \big(
(0,+\infty)\times\R\big)$.}}\end{split}\end{equation}
We claim that
\begin{equation}\label{nosn-22}
G=\R\times(-\infty,0).\end{equation}
The proof of~\eqref{nosn-22} is by contradiction. Suppose not,
then, by~\eqref{nosn-21}, we can suppose that~$
G=(-\infty,0)\times(-\infty,0)$
(the case~$G=\big( (-\infty,0)\times(-\infty,0)\big)\cup \big(
(0,+\infty)\times\R\big)$
being similar), and therefore
\begin{equation*}
{{F}}=\R\times(-\infty,0)\times(-\infty,0).\end{equation*}
This says that we can exploit Lemma~\ref{554327h:AJSJD2374859ty}
here (with~$u$ replaced by~$v$), and then deduce from~\eqref{FAG:1}
that
$$ \lim_{x_2\searrow0} u(1,x_2)=\lim_{x_2\searrow0} v(0,x_2)<0.$$
This is in contradiction with~\eqref{8hsfr}, whence the proof of~\eqref{nosn-22}
is complete.

{F}rom~\eqref{nosn-22}, it follows that
\begin{equation*}
{{F}}=\{x_3<0\}.\end{equation*}
{F}rom this and
the density estimates in~\cite{MR2675483}, we have that,
up to a subsequence, for every~$\delta>0$ there exists~$k(\delta)\in\N$
such that, if~$k\ge k(\delta)$,
\begin{equation}\label{4568rgu}
\partial (kE_v)\cap\left( [-1,1]\times\left[ \frac14,2\right]\times\R\right)\subseteq\{|x_3|\le\delta\}.
\end{equation}
Now, to complete the proof of the desired result, we
need to show that~$u\in C^1(\R^2\setminus\ell_-)$.
By~\eqref{ADX} we know that~$u\in C^1(\R\times(-\infty,0))$,
with~$\nabla u=0$ in~$\R\times(-\infty,0)$.
Moreover, by~\cite{MR3934589}, we know that~$u\in C^1(\R\times(0,+\infty))$.
Hence, to complete the proof of the desired result, it is enough to show that,
for all~$q>0$,
\begin{equation}\label{3rt7uJS:SP}
\lim_{ {x\to(q,0)}\atop{x_2>0}}|\nabla u(x)|=0.
\end{equation}
We observe that~\eqref{3rt7uJS:SP} is proved
once we demonstrate that
\begin{equation}\label{3rt7uJS:SP2}
\lim_{ {t\searrow0}}|\nabla u( 1,t)|=0.
\end{equation}
Indeed, if~\eqref{3rt7uJS:SP2} holds true,
using~\eqref{ADX6} we have that
$$ \lim_{ {x\to(q,0)}\atop{x_2>0}}|\nabla u(x)|=
\lim_{ {x\to(q,0)}\atop{x_2>0}}\left|
\nabla u\left(1,\frac{x_2}{x_1}\right)\right|=
\lim_{ {t\searrow0}}|\nabla u(1,t)|=0,
$$
which gives~\eqref{3rt7uJS:SP} in this case. 

In view of these considerations, we focus on the proof of~\eqref{3rt7uJS:SP2}.
For this, we exploit the notation in~\eqref{sucjp}, and we aim
at showing that
\begin{equation*}
\lim_{ t\searrow0}|\nabla v(0,t)|=0,
\end{equation*}
or, equivalently, that for all~$\e>0$ there exists~$k_\e\in\N$
such that if~$k\ge k_\e$ and~$t\in
\left( 0,\frac1k\right)$, we have that~$|\nabla v(0,t)|\le\e$.

For this, it is sufficient to show that
if~$k\ge k_\e$ and~$t\in
\left[ \frac{1}{k+1},\frac1k\right)$
then
\begin{equation}\label{COnsit949}
|\nabla v(0,t)|\le\e.
\end{equation}
To this end, we suppose, by contradiction, that there exists~$a>0$
such that for every~$K\in\N$ with~$K\ge1$ there exist~$k\in\N$ with~$k\ge K$
and~$t^{(k)}\in
\left[ \frac{1}{k+1},\frac1k\right)$
such that
\begin{equation}\label{P9iqdwf}|\nabla v(0,t^{(k)})|\ge a.\end{equation}
We let
\begin{equation}\label{7SH37eud} T^{(k)}:=kt^{(k)}\in
\left[ \frac{k}{k+1},1\right)\subseteq
\left[ \frac{1}{2},1\right].
\end{equation}
By~\eqref{4568rgu} and the improvement of flatness result
in~\cite{MR2675483}, choosing~$\delta$ conveniently small,
we know that, for sufficiently
large~$k$, the set~$
(kE_v)\cap\left( \left[-\frac12,\frac12\right]\times\left[ \frac12,1\right]\times\R\right)$
is the subgraph of a function~$w$, with~$|\nabla w|\le \frac{a}{2}$.
By construction,
$$ w(x)=k v\left(\frac{x}{k}\right),$$
and hence, using~\eqref{7SH37eud},
$$ \frac{a}2\ge |\nabla w(0,T^{(k)})|=\left|\nabla
v\left(0,\frac{T^{(k)}}{k}\right)\right|=|\nabla v(0,t^{(k)})|.$$
This is in contradiction with~\eqref{P9iqdwf},
and the proof of~\eqref{COnsit949}
is thereby complete.
\end{proof}

It is now convenient to take into account the ``Jacobi field''
associated to the fractional perimeter (see e.g.
formula~(1.5) in~\cite{MR3798717},
or formula~(4.30) in~\cite{SAEZ}, or Lemma~C.1
in~\cite{MR3824212}, or Section~1.3 in~\cite{MR3934589}), namely we define~$\sigma:=\frac{1+s}2$,
\begin{eqnarray*}&& {\mathcal{L}}^\sigma_E\,\eta(x):=
\int_{\partial E} \frac{\eta(y)-\eta(x)}{|y-x|^{n+2\sigma}}\,d{\mathcal{H}}^n(y),\\
&&A^\sigma_E(x):=\sqrt{
\int_{\partial E} \frac{1-\nu(y)\cdot\nu(x)}{|y-x|^{n+2\sigma}}\,d{\mathcal{H}}^n(y)}\\
{\mbox{and }}&&
{\mathcal{J}}^\sigma_E\,\eta(x):={\mathcal{L}}^\sigma_E\,\eta(x)
+\big( A^\sigma_E(x)\big)^2\,\eta(x),
\end{eqnarray*}
where~$\nu=(\nu_1,\dots,\nu_{n+1})$ is the exterior normal of~$E$
(we will often write~$\nu(x)$ to denote this normal at the point~$(x,u(x))$
if~$E=E_u$).
It is known (see Theorem~1.3(i) in~\cite{MR3934589}) that if~$E$
is $s$-minimal in~$B_r(x)$, with~$x\in\partial E$,
and~$(\partial E)\cap B_r(x)$ is of class~$C^3$, then
\begin{equation}\label{JAN-AIKS}
{\mathcal{J}}^\sigma_E\,\nu_i(x)=0,\qquad{\mbox{ for every }}i\in\{1,\dots,n+1\}.
\end{equation}
With this notation, we have the following classification result:

\begin{lemma}\label{nu-co}
Let~$u\in C^1(\R^n)$ be an $s$-minimal graph in
a domain~$\Omega\subseteq\R^n$
and assume that there exist~$i\in\{1,\dots,n+1\}$
and~$x^\star\in\Omega$ such that \begin{equation}\label{8US:1197mi}0\le
\nu_i(x^\star)\le\nu_i(x)\qquad{\mbox{
for every~$x\in\R^n$.}}\end{equation} Then~$\nu_i$ is constant
in~$\R^n$.
\end{lemma}

\begin{proof} We observe that~$u\in C^\infty(\Omega)$, due to~\cite{MR3934589},
and therefore we can exploit~\eqref{JAN-AIKS} and obtain that
\begin{eqnarray*}&& 0={\mathcal{J}}^\sigma_E\,\nu_i(x^\star)=
\int_{\partial E} \frac{\nu_i(y)-\nu_i(x^\star)}{|y-x^\star|^{n+2\sigma}}
\,d{\mathcal{H}}^n(y)
+\big( A^\sigma_E(x^\star)\big)^2\,\nu_i(x^\star)\\&&\qquad\qquad
\ge \int_{\partial E} \frac{\nu_i(y)-\nu_i(x^\star)}{|y-x^\star|^{n+2\sigma}}
\,d{\mathcal{H}}^n(y).\end{eqnarray*}
This and~\eqref{8US:1197mi} give that~$\nu_i(y)=\nu_i(x^\star)$
for every~$y\in\R^n$.
\end{proof}

{F}rom this and Lemma~\ref{STEP1} we deduce
that the boundary continuity of homogeneous
nonlocal minimal graphs give full rigidity and symmetry results:

\begin{corollary}\label{STEP11}
Let
the assumptions of Theorem~\ref{QC} hold.
Assume also that 
\begin{equation*}
\lim_{x\to (1,0)} u(x)=0=\lim_{x\to (-1,0)} u(x).
\end{equation*}
Then~$u(x)=0$ for all~$x\in\R^2$.
\end{corollary}

\begin{proof} By Lemma~\ref{STEP1}, we know that~$u\in C^1(\R^2)$.
Also, recalling~\eqref{ADX}, we have that~$\nabla u(x)=0$
if~$x_2<0$. As a consequence, we see that~$\nabla u(x)=0$
for all~$x$ with~$x_2\le0$ and therefore
\begin{equation}\label{nn}
\nu_3(x)=1 \qquad{\mbox{ for all~$x$ with~$x_2\le0$.}}
\end{equation}
Now we take~$x^\star=(x^\star_1,x^\star_2)$ such that
$$ \nu_3(x^\star)=\min_{S^1} \nu_3.$$
By~\eqref{ADX6}, we know that, for any~$x\in\R^2$,
\begin{equation}\label{7trferf274} \nu_3(x)=\nu_3\left( \frac{x}{|x|}\right)\ge\nu_3(x^\star).\end{equation}
We claim that
\begin{equation}\label{nu3}
\nu_3(x)=1\qquad{\mbox{ for all~$x\in\R^2$.}}
\end{equation}
To prove this, we distinguish two cases. If~$x^\star_2>0$,
recalling~\eqref{7trferf274},
we are in the position of using Lemma~\ref{nu-co}.
In this way we obtain that, for every~$x\in\R^2$,
$$ \nu_3(x)=\nu_3(0)=1,$$
which proves~\eqref{nu3} in this case.

If instead~$x_2^\star\le0$, we deduce from~\eqref{nn} and~\eqref{7trferf274} that,
for any~$x\in\R^2$,
$$ 1=\nu_3(x^\star)\le \nu_3(x)\le |\nu(x)|=1,$$
thus completing the proof of~\eqref{nu3}.

{F}rom~\eqref{nu3} we deduce that~$\nu(x)=(0,0,1)$ 
and hence~$\nabla u(x)=0$ for every~$x\in\R^2$,
from which we obtain the desired result.
\end{proof}

Another useful ingredient towards the proof of Theorem~\ref{QC}
consists in the following rigidity result:

\begin{lemma}\label{HAusjc293ierf}
Let~$u\in L^\infty_{\rm loc}(\R^n)$.
Let~$\Omega\subseteq\R^n$ be
a smooth and convex
domain.
Assume that~$u\in C^1(\R^n\setminus\overline{\Omega})$
and
that there exists~$x^\star\in\Omega$
such that
\begin{equation}\label{1g6778}
\partial_1 u(x^\star)\ge\partial_1 u(x)\qquad{\mbox{ for all }}
x\in \R^n\setminus (\partial\Omega).
\end{equation}
Suppose also that there exists a ball~${\mathcal{B}}\Subset
\R^n\setminus (\partial\Omega)$ such that
\begin{equation}\label{1g6778C}
\partial_1 u(x^\star)>\partial_1 u(x)\qquad{\mbox{ for all }}
x\in {\mathcal{B}}.
\end{equation}
Then,
$u$ cannot be an $s$-minimal graph in~$\Omega$.
\end{lemma}

\begin{proof} Up to a translation, we suppose that~$x^\star=0$.
To prove the desired result,
we argue for a contradiction, supposing that
$u$ is an $s$-minimal graph in~$\Omega$.
Then, by~\cite{MR3934589},
we have that~$u$ is smooth inside~$\Omega$ and thus,
by formula~(49)
in~\cite{MR3331523}, we know that, for every~$x\in\Omega$,
\begin{equation}\label{g6778}
\int_{\R^n} F\left( \frac{u(x+y)-u(x)}{|y|}\right)\,
\frac{dy}{|y|^{n+s}}=0,\end{equation}
with $$F(t):=\int_0^t \frac{d\tau}{(1+\tau^2)^{\frac{n+1+s}2}}.$$
Now, the idea that we want to implement is the following:
if one formally
takes a derivative with respect to~$x_1$ of~\eqref{g6778}
and computes it at the origin, the positivity of~$F'$ and~\eqref{1g6778}
leads to the fact that~$\partial_1u$ must be constant,
in contradiction with~\eqref{1g6778C}. Unfortunately,
this approach cannot be
implemented directly, since~$u$
is not smooth across~$\partial\Omega$ and therefore
one cannot justify the derivative of~\eqref{g6778}
under the integral sign.

To circumvent this difficulty, we argue as follows.
We let~$\e\in(0,1)$ to
be taken as small as we wish in what follows.
We also take~$\delta>0$ such that~$B_{3\delta}\subset\Omega$,
and we set
$$ G_\e(y):=F\left( \frac{u(\e e_1+y)-u(\e e_1)}{|y|}\right),$$
where~$e_1=(1,0,\dots,0)\in\R^n$. We also denote by~$G_0$ the function~$G_\e$
when~$\e=0$.

We observe that
\begin{equation}\label{adesso}
\begin{split}
G_\e(y)-G_0(y)\;=\;&F\left( \frac{u(\e e_1+y)-u(\e e_1)}{|y|}\right)-
F\left( \frac{u(y)-u(0)}{|y|}\right)\\
=\;& a_\e(y)\;
\,\frac{u(\e e_1+y)-u(\e e_1)-u(y)+u(0)}{|y|},
\end{split}\end{equation}
where
$$ a_\e(y):=
\int_0^1 F'\left( \frac{t[u(\e e_1+y)-u(\e e_1)]+(1-t)[u(y)-u(0)]}{|y|}\right)\,dt\in(0,1].$$
We remark that
\begin{equation}\label{VUSJDrg}
\lim_{\e\searrow0}a_\e(y)=a_0(y):=
F'\left( \frac{ u(y)-u(0) }{|y|}\right).\end{equation}
Furthermore, we observe that, if~$\e\in(0,\delta)$,
\begin{equation}\label{2347gBS12}\Big|
u(\e e_1)-u(0)- \e\partial_1 u(0)\Big|\le\e^2\| u\|_{C^2(B_{\delta})}.
\end{equation}
Now, we let~${\mathcal{R}}:=\R^n\setminus (\partial\Omega)$
and we
distinguish two regions of space, namely if~$y\in {\mathcal{R}}\setminus B_\delta$
and if~$y\in B_\delta$. Firstly, if~$y\in {\mathcal{R}}\setminus B_\delta$,
we consider the segment joining~$y$ to~$\e e_1+y$,
and we observe that it meets~$\partial\Omega$ in at most
one point, in light of the convexity of the domain.
That is, we have that~$y+\tau e_1\in {\mathcal{R}}$
for every~$\tau\in[0,\tau_\e)\cup(\tau_\e,\e)$
for a suitable~$\tau_\e\in(0,\e]$
(with the notation that when~$\tau_\e=\e$
the set~$(\tau_\e,\e)$ is empty). Then, we have that
\begin{eqnarray*}
u(\e e_1+y)-u(y)&=&u(y+\e e_1)-
u(y+\tau_\e e_1)+u(y+\tau_\e e_1)
-u(y)\\&=&
\int_{\tau_\e}^\e \partial_1 u(y+\tau e_1)\,d\tau+
\int^{\tau_\e}_0 \partial_1 u(y+\tau e_1)\,d\tau\\&\le&
\int_{\tau_\e}^\e \partial_1 u(0)\,d\tau+
\int^{\tau_\e}_0 \partial_1 u(0)\,d\tau\\&=&\e\partial_1 u(0),
\end{eqnarray*}
thanks to~\eqref{1g6778}.
Therefore, for every~$y\in {\mathcal{R}}\setminus B_\delta$,
recalling~\eqref{adesso} and~\eqref{2347gBS12}
we have that
\begin{eqnarray*}
G_\e(y)-G_0(y)&\le&
a_\e(y)\;
\,\frac{\e\partial_1 u(0)-u(\e e_1)+u(0)}{|y|}\\
&\le&
\frac{\e^2\| u\|_{C^2(B_{\delta})}}{|y|}.
\end{eqnarray*}
Consequently,
\begin{equation}\label{JSY:001}
\int_{ {\mathcal{R}}\setminus (B_\delta\cup{\mathcal{B}}) }\frac{
G_\e(y)-G_0(y)}{\e}\,\frac{dy}{|y|^{n+s}}\le
\e\,\| u\|_{C^2(B_{\delta})}\,
\int_{ {\mathcal{R}}\setminus B_\delta }\frac{dy}{|y|^{n+s+1}}
\le \frac{C\,\e\,\| u\|_{C^2(B_{\delta})}}{\delta^{s+1}},
\end{equation}
for some~$C>0$.

Furthermore, if~$y\in{\mathcal{B}}\setminus B_\delta$,
we recall~\eqref{2347gBS12} and we write that
\begin{eqnarray*}
u(\e e_1+y)-u(y)-u(\e e_1)+u(0)&\le& \e
\int_0^1\partial_1u(\tau \e e_1+y)\,d\tau-
\e\partial_1 u(0)+\e^2\| u\|_{C^2(B_{\delta})}\\
&=&\e
\int_0^1\big(\partial_1 u(\tau \e e_1+y)-
\partial_1 u(0)\big)\,d\tau+\e^2\| u\|_{C^2(B_{\delta})}.
\end{eqnarray*}
This gives that
\begin{equation}\label{JSY:002}
\begin{split}&
\int_{ {\mathcal{B}}\setminus B_\delta }\frac{
G_\e(y)-G_0(y)}{\e}\,\frac{dy}{|y|^{n+s}}\\ \le\;&\int_{ {\mathcal{B}}\setminus B_\delta }\left[
\int_0^1\big(\partial_1 u(\tau \e e_1+y)-
\partial_1 u(0)\big)\,d\tau\right]\,\frac{a_\e(y)\,dy}{|y|^{n+s+1}}+\frac{C\,
\e\| u\|_{C^2(B_{\delta})}}{ \delta^{s+1}}
,\end{split}\end{equation}
up to renaming~$C>0$.

Now we focus on the case~$y\in B_\delta$.
In this case, 
using that~$F'$ is even, we see that
\begin{eqnarray*}
&& G_\e(y)+G_\e(-y)-G_0(y)-G_0(-y)\\
&=& F\left( \frac{u(\e e_1+y)-u(\e e_1)}{|y|}\right)+F\left( \frac{u(\e e_1-y)-u(\e e_1)}{|y|}\right)
\\&&\qquad-F\left( \frac{u(y)-u(0)}{|y|}\right)-F\left( \frac{u(-y)-u(0)}{|y|}\right)
\\ &=&\int_0^\e
\Bigg[
F'\left( \frac{u(t e_1+y)-u(t e_1)}{|y|}\right)
\frac{\partial_1 u(t e_1+y)-\partial_1 u(t e_1)}{|y|}\\&&\qquad+
F'\left( \frac{u(t e_1-y)-u(t e_1)}{|y|}\right)
\frac{\partial_1 u(t e_1-y)-\partial_1 u(t e_1)}{|y|}
\Bigg]\,dt\\ &\le&\int_0^\e
\Bigg[
\Bigg| F'\left( \frac{u(t e_1+y)-u(t e_1)}{|y|}\right)\Bigg|\;
\frac{|\partial_1 u(t e_1+y)
+\partial_1 u(t e_1-y)-2\partial_1 u(t e_1)|
}{|y|}\\&&\qquad+\Bigg|F'\left( \frac{u(t e_1+y)-u(t e_1)}{|y|}\right)-
F'\left( \frac{u(t e_1-y)-u(t e_1)}{|y|}\right)\Bigg|\;
\frac{|\partial_1 u(t e_1-y)-\partial_1 u(t e_1)|}{|y|}
\Bigg]\,dt\\ &\le& C\,\int_0^\e
\Bigg[
\frac{|\partial_1 u(t e_1+y)
+\partial_1 u(t e_1-y)-2\partial_1 u(t e_1)|
}{|y|}\\&&\qquad+\Bigg|F'\left( \frac{u(t e_1+y)-u(t e_1)}{|y|}\right)-
F'\left( \frac{u(t e_1)-u(t e_1-y)}{|y|}\right)\Bigg|\;
\frac{|\partial_1 u(t e_1-y)-\partial_1 u(t e_1)|}{|y|}
\Bigg]\,dt
\\&\le& C\,\int_0^\e
\Bigg[
\frac{\|u\|_{C^3(B_{2\delta})}\,|y|^2
}{|y|}+
\frac{|u(t e_1+y)+u(t e_1-y)-2u(t e_1)|}{|y|}
\;\frac{|\partial_1 u(t e_1-y)-\partial_1 u(t e_1)|}{|y|}
\Bigg]\,dt\\&\le& C\,\int_0^\e
\Bigg[
\|u\|_{C^3(B_{2\delta})}\,|y|+
\frac{\|u\|_{C^2(B_{2\delta})}^2\,|y|^2}{|y|}
\Bigg]\,dt\\
&\le& C\,\big(1+\|u\|_{C^3(B_{2\delta})}^2\big)\,\e\,|y|,
\end{eqnarray*}
up to renaming~$C$.

As a consequence,
\begin{equation}\label{JSY:003}\begin{split}&
\int_{ B_\delta }\frac{
G_\e(y)-G_0(y)}{\e}\,\frac{dy}{|y|^{n+s}}
=\frac12
\int_{ B_\delta }\frac{
\big(G_\e(y)-G_0(y)\big)+\big(
G_\e(-y)-G_0(-y)\big)
}{\e}\,\frac{dy}{|y|^{n+s}}\\&\qquad
\le 
C\,\big(1+\|u\|_{C^3(B_{2\delta})}^2\big)\,
\int_{ B_\delta }\frac{dy}{|y|^{n+s-1}}=C\,\big(1+\|u\|_{C^3(B_{2\delta})}^2\big)\,\delta^{1-s},
\end{split}
\end{equation}
for some~$C>0$ independent of~$\e$ and~$\delta$.

Now we combine~\eqref{JSY:001}, \eqref{JSY:002} and~\eqref{JSY:003}
with~\eqref{g6778}. In this way, we find that
\begin{eqnarray*}
0&=&\frac1\e\left[
\int_{\R^n} F\left( \frac{u(\e e_1+y)-u(\e e_1)}{|y|}\right)\,
\frac{dy}{|y|^{n+s}}-
\int_{\R^n} F\left( \frac{u(y)-u(0)}{|y|}\right)\,
\frac{dy}{|y|^{n+s}}\right]\\&=&
\int_{\R^n}\frac{
G_\e(y)-G_0(y)}{\e}\,\frac{dy}{|y|^{n+s}}\\
&=&
\int_{{\mathcal{R}}\setminus( B_\delta\cup{\mathcal{B}})}\frac{
G_\e(y)-G_0(y)}{\e}\,\frac{dy}{|y|^{n+s}}+
\int_{B_\delta}\frac{
G_\e(y)-G_0(y)}{\e}\,\frac{dy}{|y|^{n+s}}\\&&\qquad
+
\int_{{\mathcal{B}}\setminus B_\delta}\frac{
G_\e(y)-G_0(y)}{\e}\,\frac{dy}{|y|^{n+s}}\\&\le&
\int_{ {\mathcal{B}}\setminus B_\delta }\left[
\int_0^1\big(\partial_1 u(\tau \e e_1+y)-
\partial_1 u(0)\big)\,d\tau\right]\,\frac{a_\e(y)\,dy}{|y|^{n+s+1}}\\
&&\qquad+\frac{2C\,
\e\| u\|_{C^2(B_{\delta})}}{\delta^{s+1}}+C\,\big(1+\|u\|_{C^3(B_{2\delta})}^2\big)\,\delta^{1-s}.
\end{eqnarray*}
For this reason, and recalling~\eqref{VUSJDrg},
when we take the limit as~$\e\searrow0$ we see that
$$ 0\le
\int_{ {\mathcal{B}}\setminus B_\delta }\big(\partial_1 u(y)-
\partial_1 u(0)\big)\,\frac{a_0(y)\,dy}{|y|^{n+s+1}}+C\,\big(1+\|u\|_{C^3(B_{2\delta})}^2\big)\,\delta^{1-s}
.$$
Now we can take the limit as~$\delta\searrow0$
and conclude that
$$ 0\le
\int_{ {\mathcal{B}} }\big(\partial_1 u(y)-
\partial_1 u(0)\big)\,\frac{a_0(y)\,dy}{|y|^{n+s+1}}
.$$
{F}rom this and~\eqref{1g6778C}
a contradiction plainly follows.
\end{proof}

The proof of Theorem~\ref{QC} will also rely on the following simple, but interesting,
calculus observation:

\begin{lemma}\label{87923-4656}
Let~$\delta>0$ and~$g\in C([0,\delta])\cap C^1((0,\delta])$.
Assume that
\begin{equation}\label{7US-dc}
\lim_{t\searrow0} g'(t)=+\infty
\end{equation}
and let
\begin{equation}\label{grty0}
V(t):=g(t)-tg'(t).
\end{equation}
Suppose also that the following limit exists
\begin{equation} \label{8903456wdf}
\R\cup\{-\infty\}\ni V(0):=\limsup_{t\searrow0} V(t)\le g(0).\end{equation}
Then there exists~$t_0\in(0,\delta)$ such that
\begin{equation}\label{jva7823}V(t_0)> V(0).\end{equation}
In particular, there exists~$t_\star\in(0,\delta]$ such that
\begin{equation}\label{jva7823BIS} \max_{t\in[0,\delta]} V(t)=V(t_\star).\end{equation}
\end{lemma}

\begin{proof} For every~$r\ge0$ we consider the straight line
$$ R_r(t):= \frac{\big(g(\delta)-g(0)\big)\,t}{\delta}+g(0)+r.$$
We observe that, if~$r\ge 4\|g\|_{L^\infty([0,\delta])}+1$,
for all~$t\in[0,\delta]$,
$$ R_r(t)-g(t)\ge-
\frac{2\|g\|_{L^\infty([0,\delta])}\,t}{\delta}-\|g\|_{L^\infty([0,\delta])}+r-\|g\|_{L^\infty([0,\delta])}
\ge r-4\|g\|_{L^\infty([0,\delta])}\ge1.$$
Hence, we take~$r_\star\ge0$ to be the smallest~$r$
for which~$R_r(t)-g(t)\ge0$ for all~$t\in[0,\delta]$.

We claim that
\begin{equation}\label{rsya}
r_\star>0.
\end{equation}
Indeed, if not, we have that, for all~$r\ge0$ and all~$t\in(0,\delta]$,
\begin{eqnarray*}
0\le \frac{R_r(t)-g(t)}t=\frac{ g(\delta)-g(0) }{\delta}+ \frac{g(0)+r-g(t)}t,
\end{eqnarray*}
and thus
$$ 0\le\frac{ g(\delta)-g(0) }{\delta}+ \frac{g(0)-g(t)}t=
\frac{ g(\delta)-g(0) }{\delta}- \frac{1}t\int_0^t g'(\tau)\,d\tau
.$$
Recalling~\eqref{7US-dc}, we have that for any~$M>0$ there exists~$t_M>0$
such that, for every~$t\in(0,t_M)$, $g'(t)\ge M$ and consequently
$$ M\le\frac{1}{t_M}\int_0^{t_M} g'(\tau)\,d\tau\le \frac{ g(\delta)-g(0) }{\delta}.$$
This is a contradiction if~$M$ is taken sufficiently large, and hence
the proof of~\eqref{rsya} is complete.

Then, there exists~$t_0\in[0,\delta]$ such that~$R_{r_\star}(t_0)=g(t_0)$.
In addition, since, by~\eqref{rsya},
\begin{eqnarray*}&& R_{r_\star}(0)=g(0)+r_\star>g(0)\\
\quad{\mbox{and }}&&
R_{r_\star}(\delta)= \frac{\big(g(\delta)-g(0)\big)\,\delta}{\delta}+g(0)+r_\star
=g(\delta)+r_\star>g(\delta)
,\end{eqnarray*}
we have that~$t_0\not\in\{0,\delta\}$. 

We show now that $t_0$ satisfies~\eqref{jva7823}. For this,
using that~$R_{r_\star}(t_0)-g(t_0)=0\le R_{r_\star}(t)-g(t)$
for every~$t\in[0,\delta]$, it follows that $$
0=
R_{r_\star}'(t_0)-g'(t_0)=
\frac{g(\delta)-g(0)}{\delta}-g'(t_0).
$$
This and~\eqref{grty0}
say that
\begin{eqnarray*}&& V(t_0)= g(t_0)-t_0\,g'(t_0)=
g(t_0)-
\frac{\big(g(\delta)-g(0)\big)\,t_0}{\delta}\\&&\qquad
=g(t_0)-R_{r_\star}(t_0)+g(0)+r_\star=g(0)+r_\star>g(0).
\end{eqnarray*}
This, together with~\eqref{8903456wdf},
proves~\eqref{jva7823}, which in turn implies~\eqref{jva7823BIS}.
\end{proof}

Now, we can complete the proof of Theorem~\ref{QC}
in dimension~$n:=2$
(the case~$n:=1$ being already covered by Theorem~4.1
in~\cite{2019arXiv190405393D}), by arguing as follows.

\begin{proof}[Proof of Theorem~\ref{QC} when $n:=2$]
By~\eqref{ADX0} in Lemma~\ref{lemma33}, we can define, for all~$x_1\in[-1,1]$,
\begin{equation}\label{uias23} u(x_1,0):=\lim_{{(y_1,y_2)\to(x_1,0)}\atop{y_2>0}}u(y_1,y_2).\end{equation}
Notice that if~$u(-1,0)=u(1,0)=0$, then the desired
result follows from Corollary~\ref{STEP11}.
Hence, without loss of generality, we can assume that
\begin{equation}\label{Qnd}
(0,+\infty)\ni u(1,0)\ge |u(-1,0)|.
\end{equation}
We claim that such case cannot hold, by reaching a contradiction.
To this end, in view of~\eqref{Qnd} and~\cite{MR3532394}, we have that, in a small neighborhood of~$P=(P_1,P_2,P_3):=(1,0,u(1,0))$,
one can write
\begin{equation}\label{UNAJdf}
{\mbox{$(\partial E_u)\cap\{x_2>0\}$ as the graph of a $C^{1,\frac{1+s}2}$-function in the $e_2$-direction,}}
\end{equation} that is,
in the vicinity of the point~$P$, the set~$\{x_{3}=u(x)\}$
in~$\{x_2>0\}$
coincides with the set~$\{x_2=v(x_1,x_{3})\}$, for a suitable~$v\in
C^{1,\frac{1+s}2}(\R^2)$, with
\begin{equation}\label{x2x0}
{\mbox{$v(P_1,x_3)=0$
if~$x_3\le P_3$, and $v(P_1,x_3)\ge0$
if~$x_3\ge P_3$.}}\end{equation}
Hence, for~$x$ close to~$P$ with~$x_2>0$, we can write that
\begin{equation}\label{x2x1} x_2=v\big(x_1,u(x_1,x_2)\big).\end{equation}

Now, we let~$\delta_0\in(0,1]$ to be taken conveniently small in what
follows, we define
$${\mathcal{Q}}:=\Big\{x=(x_1,x_2)\in\R^2 {\mbox{ s.t. either $|x_1|=1$
and~$0<|x_2|\le\delta_0$,
or $|x_1|\leq1$
and~$|x_2|=\delta_0$}}\Big\},$$ and we
claim that there exists~$x^\star=(x^\star_1,x^\star_2)$
such that~$x^\star_2>0$ and
\begin{equation}\label{2347we}
\partial_1 u(x^\star)=\max_{{\mathcal{Q}}}\partial_1 u\in(0,+\infty).
\end{equation}
For this, we define
\begin{equation}\label{la-ggad}
g(t):=u(1,t)\qquad{\mbox{and}}\qquad V(t):=\partial_1 u(1,t).\end{equation}
We use~\eqref{x2x1} and the smoothness of~$u$ in~$\{x_2>0\}$ (see~\cite{MR3934589}), to see that, if~$x_2>0$ is sufficiently small,
$$ 0 =\partial_2\Big( x_2-v\big(1,u(1,x_2)\big)\Big)
= 1-\partial_3 v\big(1,u(1,x_2)\big)\,\partial_2 u(1,x_2).$$
In particular, we have that
\begin{eqnarray}
&&
\label{pap3}\partial_3 v\big(1,u(1,x_2)\big)\ne0,\\
\label{pap5a}&&\partial_2 u(1,x_2)\ne 0\\
\label{pap5}
{\mbox{and }}&& \partial_2 u(1,x_2)=
\frac{1}{ \partial_3 v\big(1,u(1,x_2)\big) }
\end{eqnarray}
if~$x_2>0$ is sufficiently small.

By~\eqref{x2x0} and~\eqref{pap3}, we conclude that,
if~$x_2>0$ is sufficiently small,
\begin{equation*}\partial_3 v\big(1,u(1,x_2)\big)>0.\end{equation*}
This, \eqref{x2x0} and~\eqref{pap5} give that
\begin{equation}\label{degeru7443y}
\lim_{x_2\searrow0} \partial_2 u(1,x_2)=\lim_{x_3\searrow P_3}
\frac{1}{ \partial_3 v\big(1,x_3\big) } =+\infty,\end{equation}
that is, in the notation of~\eqref{la-ggad},
\begin{equation}\label{LLxv23e0}
\lim_{t\searrow0}g'(t)=+\infty.
\end{equation}

Also, by \eqref{ADX6}, we have that~$ u(x)=\nabla u(x)\cdot x$, and consequently
\begin{equation} \label{alsousing}
u(1,x_2)=\partial_1u(1,x_2)+x_2\partial_2u(1,x_2).\end{equation}
Thus, recalling the notation in~\eqref{la-ggad}, if~$t>0$ is sufficiently small,
\begin{equation}\label{LLxv23e1}
V(t)-g(t)+tg'(t)=\partial_1 u(1,t)-u(1,t)+t\partial_2 u(1,t)
=0.
\end{equation}
Using~\eqref{degeru7443y}, \eqref{alsousing} and the notation of~\eqref{la-ggad}, we see that
$$ V(t)=\partial_1 u(1,t)= u(1,t)-t\partial_2 u(1,t)\le u(1,t)=g(t),
$$ 
and consequently
\begin{equation}\label{LLxv23e} \R\cup\{-\infty\}\ni \limsup_{t\searrow0}V(t)\le
g(0).\end{equation}
Now, in light of~\eqref{LLxv23e0}, \eqref{LLxv23e1} and~\eqref{LLxv23e},
we see that conditions~\eqref{7US-dc},
\eqref{grty0} and~\eqref{8903456wdf}
are satisfied in this setting. Consequently, we can exploit Lemma~\ref{87923-4656}
and deduce from~\eqref{jva7823BIS} that
there exists~$t_\star\in(0,\delta_0]$ such that, for all~$t\in[0,\delta_0]$,
\begin{equation}\label{365}
\partial_1u(1,t)= V(t)\le V(t_\star)=\partial_1u(1,t_\star)=\max_{{\mathcal{Q}}\cap\{x_2>0\}} \partial_1u.
\end{equation}
Now we claim that
\begin{equation}\label{o3ra3s2i234}
\max_{{\mathcal{Q}}\cap\{x_2>0\}} \partial_1u>0.
\end{equation}
Indeed, if not, we have that~$\partial_1u(1,t)\le 0$,
as long as~$t\in(0,\delta_0)$. In particular, if~$\tau\in(0,\delta_0^2)$
and~$\sigma\in(\sqrt{\tau},1)$, we have that~$\frac\tau\sigma\le \sqrt\tau<\delta_0$
and, as a result, exploiting~\eqref{ADX6},
\begin{eqnarray*}&& u(1,\tau)-u(\sqrt{\tau},\tau)=
\int_{ \sqrt{\tau} }^1 \partial_1 u(\sigma,\tau)\,d\sigma=
\int_{ \sqrt{\tau} }^1 \partial_1 u\left(1,\frac\tau\sigma\right)\,d\sigma\le0.\end{eqnarray*}
Consequently, taking the limit as~$\tau\searrow0$ and recalling~\eqref{ADX3}
and~\eqref{uias23},
we conclude that~$u(1,0)\le0$. This is in contradiction with~\eqref{Qnd}
and thus this completes the proof of~\eqref{o3ra3s2i234}.

Since~$\partial_1 u=0$ in~$\{x_2<0\}$, from~\eqref{365}
and~\eqref{o3ra3s2i234},
we have that 
$$ (0,+\infty)\ni\partial_1u(1,t_\star)=\max_{{\mathcal{Q}}} \partial_1u,$$
and hence~\eqref{2347we} follows directly .

This, together with~\eqref{ADX}, allows us to exploit Lemma~\ref{HAusjc293ierf},
used here with~$\Omega:=\{x\in(0,2)\times(0,+\infty)
{\mbox{ s.t. }}\frac{x_2}{x_1}<\delta_0\}$, ${\mathcal{B}}:=B_1(-2,-2)$ and~$x^\star:=
(1,t_\star)$, and this
yields that~$u$ cannot be $s$-minimal in~$\Omega\subset\{x_2>0\}$.
This
is a contradiction with our assumptions, and hence 
the setting in~\eqref{Qnd} cannot occur.

As a result, we have that~$u(x_1,0)=0$ for all~$x_1\in\R$.
Then, in view of Corollary~\ref{STEP11},
we conclude that~$u$ vanishes identically, and this completes
the proof of Theorem~\ref{QC}.
\end{proof}

It is interesting to point out that, as a byproduct of Theorem~\ref{QC},
one also obtains the following alternatives on the second blow-up:

\begin{corollary}\label{COR:ALTEzz}
Let~$u$ be an $s$-minimal graph in~$(-2,2)\times(0,4)$.
Assume that there exists~$h>0$
such that~$u=0$ in~$ (-2,2)\times(-h,0)$.

Let~$E_{00}$ be the second blow-up of~$E_u$.
Then:
\begin{eqnarray}\label{POS-E001-BI1}
&& {\mbox{either $E_{00}\cap\{x_2>0\}=\varnothing$,}}
\\ \label{POS-E001-BI2}&& {\mbox{or $E_{00}\cap\{x_2>0\}=\{x_2>0\}$,}}\\&&
\label{POS-E001-BI3}{\mbox{or $E_{00}=\{x_{3}<0\}$.}}
\end{eqnarray}
\end{corollary}

\begin{proof} We assume that neither~\eqref{POS-E001-BI1}
nor~\eqref{POS-E001-BI2} hold true, and we prove that~\eqref{POS-E001-BI3}
is satisfied. For this, we first exploit Lemma~\ref{BOAmdfiUAMP},
deducing from~\eqref{POS-E003} and~\eqref{POS-E004}
that~$E_{00}$ has a graphical
structure, with respect to some function~$u_{00}$, satisfying
\begin{equation}\label{ADX662738495jdjfgjg-2}
\lim_{x\to0}u_{00}(x)=0.
\end{equation}
Moreover, we know that~$E_{00}$ is a homogeneous set
(see e.g. Lemma 2.2 in~\cite{2019arXiv190405393D}),
and thus
\begin{equation}\label{ADX662738495jdjfgjg}
u_{00}(tx)=tu_{00}(x)\qquad{\mbox{ for all $x\in\R^2$ and~$t>0$.}}
\end{equation}
In view of~\eqref{ADX662738495jdjfgjg-2}
and~\eqref{ADX662738495jdjfgjg},
we are in the position of applying Theorem~\ref{QC}
to the function~$u_{00}$, and thus we conclude that~$u_{00}$
vanishes identically, and this establishes~\eqref{POS-E001-BI3}.
\end{proof}

We also observe that Corollary~\ref{COR:ALTEzz}
can be further refined in light of Lemma~\ref{554327h:AJSJD2374859ty}.

\begin{corollary}\label{COR:ALTE}
Let~$u$ be an $s$-minimal graph in~$(-2,2)\times(0,4)$.
Assume that there exists~$h>0$
such that~$u=0$ in~$ (-2,2)\times(-h,0)$.

Let~$E_{00}$ be the second blow-up of~$E_u$.
Then:
\begin{eqnarray}\label{POS-E001-BI1b}
&& {\mbox{either $E_{00}\cap\{x_2>0\}=\varnothing$
and }}\lim_{x_2\searrow0} u(0,x_2)<0,
\\ \label{POS-E001-BI2b}&& {\mbox{or $E_{00}\cap\{x_2>0\}=\{x_2>0\}$
and }}
\lim_{x_2\searrow0} u(0,x_2)>0,
\\&&
\label{POS-E001-BI3b}{\mbox{or $E_{00}=\{x_{3}<0\}$.}}
\end{eqnarray}
\end{corollary}

\begin{proof} In the setting of Corollary~\ref{COR:ALTEzz},
if~\eqref{POS-E001-BI1} holds true,
then we exploit~\eqref{FAG:1} and~\eqref{POS-E001-BI1b}
readily follows.

Similarly, combining~\eqref{POS-E001-BI2} with~\eqref{FAG:2},
we obtain~\eqref{POS-E001-BI2b}.

Finally, the situation in~\eqref{POS-E001-BI3} coincides with that in~\eqref{POS-E001-BI3b},
thus exhausting all the available possibilities.
\end{proof}

\section{Useful barriers}\label{HDB}

In this section we construct an auxiliary barrier,
that we will exploit in the proof of Theorem~\ref{NLMS}
to rule out the case of boundary Lipschitz
singularities. For this, we will rely on a special function introduced
in Lemma~7.1 of~\cite{2019arXiv190405393D} and on a codimension-one
auxiliary construction (given the possible use
of such barriers in other context, we give our construction
in a general dimension~$n$, but we will then restrict to the case~$n=2$
when dealing with the proof of our main theorems,
see Open Problem~\ref{17}).

These barriers rely on a purely nonlocal feature,
since they present a corner at the origin,
which maintains a significant influence on
the nonlocal mean curvature in a full neighborhood
(differently from the classical case,
in which the mean curvature is a local operator).

To perform our construction, we recall the definition of the nonlocal mean
curvature in~\eqref{jfgnjbj96768769}
and we first show that flat higher dimensional
extensions preserve the nonlocal mean curvature, up to a multiplicative
constant:

\begin{lemma}\label{5444542e62734}
Let~$G\subset\R^2$ and~$(a,b)\in\partial G$. Let
$$G^\star:=\{(x_1,\dots,x_n,x_{n+1})\in\R^{n+1}
{\mbox{ s.t. }} (x_n,x_{n+1})\in G\}.$$
Then, for every~$p=(p_1,\dots,p_n,p_{n+1})\in\R^{n+1}$ with~$(p_n,p_{n+1})=(a,b)$,
$$ {\mathcal{H}}^s_{G^\star}(p)=C(n,s)\,{\mathcal{H}}^s_{G}(a,b),$$
for a suitable constant~$C(n,s)>0$.
\end{lemma}

\begin{proof} Using the notation~$y'=(y_1,\dots,y_{n-1})$
and the change of variable
$$ \zeta':=\frac{y'-p'}{\sqrt{|y_n-p_n|^2+|y_{n+1}-p_{n+1}|^2}},$$
we have that
\begin{eqnarray*}
{\mathcal{H}}^s_{G^\star}(p)&=&\int_{\R^{n+1}}\frac{\chi_{\R^{n+1}\setminus G^\star}(y)-
\chi_{G^\star}(y)
}{|y-p|^{n+1+s}}\,dy\\&=&
\int_{\R^{n+1}}\frac{\chi_{\R^2\setminus G}(y_n,y_{n+1})-
\chi_{G}(y_n,y_{n+1})
}{\big(|y'-p'|^2 +|y_n-p_n|^2+|y_{n+1}-p_{n+1}|^2\big)^{\frac{n+1+s}2}}\,dy\\&=&
\int_{\R^2}\left[
\int_{\R^{n-1}}\frac{\chi_{\R^2\setminus G}(y_n,y_{n+1})-
\chi_{G}(y_n,y_{n+1})
}{\big(1+|\zeta'|^2\big)^{\frac{n+1+s}2}\,
\big(|y_n-a|^2+|y_{n+1}-b|^2\big)^{\frac{2+s}2}}\,d\zeta'\right]\,d(y_n,y_{n+1})\\
&=& \int_{\R^{n-1}}\frac{d\zeta'}{\big(1+|\zeta'|^2\big)^{\frac{n+1+s}2}}\;{\mathcal{H}}^s_{G}(a,b),
\end{eqnarray*}
which gives the desired result.
\end{proof}

While Lemma~\ref{5444542e62734} deals with a flat higher dimensional
extension of a set, we now turn our attention to the case
in which a higher dimensional extension is obtained by a
given function~$\Phi\in C^{1,\vartheta}(\R^{n-1})$,
with~$\vartheta\in(s,1]$. In this setting, using the notation~$x'=(x_1,\dots,x_{n-1})$,
given~$L>0$ and a function~$\beta:(-L,L)\to\R$, for every~$\varpi\ge0$
we define
$$ E^{(\beta,\varpi)}:= \big\{
x=(x_1,\dots,x_n,x_{n+1})\in\R^{n+1}{\mbox{ s.t. }}x_{n+1}<\beta(x_n)-\varpi\,\Phi(x')
\big\}.$$
We also extend~$\beta$ to take value equal to~$-\infty$ outside~$(-L,L)$.
Then, recalling the framework in~\eqref{LEYUDEF},
we can estimate the nonlocal mean curvature
of~$E^{(\beta,\varpi)}$ with that of~$E_\beta$ as follows:

\begin{lemma}\label{VGbjatrtta9}
Let~$\gamma>0$.
If~$p=(p_1,\dots,p_n,p_{n+1})\in\partial E^{(\beta,\varpi)}$
with~$p_n\in(-L,L)$,
$\beta\in C^{1,1}\big([p_n-\gamma,p_n+\gamma]\big)$,
and~$C(n,s)>0$ is as in Lemma~\ref{5444542e62734}, we have that
$$ \big| {\mathcal{H}}^s_{E^{(\beta,\varpi)}}(p)-C(n,s)\,
{\mathcal{H}}^s_{E_\beta}(p_n,p_{n+1})\big|\le C\,\varpi\,(1+\varpi),$$
for a suitable constant~$C>0$ depending only on~$n$, $s$, $L$,
$\gamma$, $\Phi$
and~$\beta$.
\end{lemma}

\begin{proof} We use the notation~$x=(x_1,\dots,x_n)$
and~$\beta_\varpi(x):=\beta(x_n)-\varpi\,\Phi(x')$.
We recall~\eqref{g6778} to write that
\begin{equation}\label{923461d52s99tgv7}
\begin{split}
2{\mathcal{H}}^s_{E^{(\beta,\varpi)}}(p)\,=\,&
2{\mathcal{H}}^s_{E_{\beta_\varpi}}(p)
\\ =\,&
\int_{\R^n} \left[
F\left( \frac{\beta_\varpi(p+y)-\beta_\varpi(p)}{|y|}\right)+
F\left( \frac{\beta_\varpi(p-y)-\beta_\varpi(p)}{|y|}\right)\right]
\,
\frac{dy}{|y|^{n+s}}\\
=\,&
\int_{\R^n} \left[
F\left( \frac{\beta(p_n+y_n)-\beta(p_n)}{|y|}-\frac{\varpi \big(\Phi(p'+y')
-\Phi(p')\big)}{|y|}\right)\right.\\
&\qquad+\left.
F\left( \frac{\beta(p_n-y_n)-\beta(p_n)}{|y|}-\frac{\varpi \big(\Phi(p'-y')
-\Phi(p')\big)}{|y|}\right)\right]
\,
\frac{dy}{|y|^{n+s}}\,.
\end{split}\end{equation}
Now we define
\begin{eqnarray*} \Psi(\varpi)&:=&
F\left( \frac{\beta(p_n+y_n)-\beta(p_n)}{|y|}-\frac{\varpi \big(\Phi(p'+y')
-\Phi(p')\big)}{|y|}\right)
\\&&\qquad+F\left( \frac{\beta(p_n-y_n)-\beta(p_n)}{|y|}-\frac{\varpi 
\big(\Phi(p'-y')
-\Phi(p')\big)}{|y|}\right).\end{eqnarray*}
We remark that
\begin{eqnarray*} |\Psi'(\varpi)|&=&\Bigg|
F'\left( \frac{\beta(p_n+y_n)-\beta(p_n)}{|y|}-\frac{\varpi \big(\Phi(p'+y')
-\Phi(p')\big)}{|y|}\right)
\frac{\Phi(p'+y')
-\Phi(p')}{|y|}
\\&&\qquad+F'\left( \frac{\beta(p_n-y_n)-\beta(p_n)}{|y|}-\frac{\varpi
\big(\Phi(p'-y')
-\Phi(p')\big)}{|y|}\right)
\frac{\Phi(p'-y')
-\Phi(p')}{|y|}\Bigg|
\\&\le&
\left|F'\left( \frac{\beta(p_n+y_n)-\beta(p_n)}{|y|}-\frac{\varpi \big(\Phi(p'+y')
-\Phi(p')\big)}{|y|}\right)\right|\\
&&\quad\times\left|
\frac{\Phi(p'+y')
-\Phi(p')}{|y|}+\frac{\Phi(p'-y')
-\Phi(p')}{|y|}\right|\\
&&+
\left|F'\left( \frac{\beta(p_n+y_n)-\beta(p_n)}{|y|}-\frac{\varpi \big(\Phi(p'+y')
-\Phi(p')\big)}{|y|}\right)\right.\\
&&\quad\left.-F'\left( \frac{\beta(p_n+y_n)-\beta(p_n)}{|y|}-\frac{\varpi \big(\Phi(p'-y')
-\Phi(p')\big)}{|y|}\right)\right|\;\left|
\frac{\Phi(p'-y')
-\Phi(p')}{|y|}\right|
.
\end{eqnarray*}
We also use that~$F'$ is even to observe that
\begin{eqnarray*}&&
\Bigg|
F'\left( \frac{\beta(p_n+y_n)-\beta(p_n)}{|y|}-\frac{\varpi
\big(\Phi(p'+y')
-\Phi(p')\big)}{|y|}\right)\\
&&\qquad-F'\left( \frac{\beta(p_n-y_n)-\beta(p_n)}{|y|}-\frac{\varpi \big(\Phi(p'-y')
-\Phi(p')\big)}{|y|}\right)
\Bigg|\\
&=&
\Bigg|
F'\left( \frac{\beta(p_n+y_n)-\beta(p_n)}{|y|}-\frac{\varpi
\big(\Phi(p'+y')
-\Phi(p')\big)}{|y|}\right)\\
&&\qquad-F'\left( \frac{\beta(p_n)-\beta(p_n-y_n)}{|y|}+\frac{\varpi
\big(\Phi(p'-y')
-\Phi(p')\big)}{|y|}\right)
\Bigg|\\
&\le& \min\Bigg\{2\,\| F'\|_{L^\infty(\R)},\\&&\quad
\| F''\|_{L^\infty(\R)}\left(
\frac{|\beta(p_n+y_n)+\beta(p_n-y_n)-2\beta(p_n)|}{|y|}+\varpi\,
\frac{\big|\Phi(p'+y')+\Phi(p'-y')
-2\Phi(p')\big|}{|y|}
\right)
\Bigg\}\\
&\le& C\,(1+\varpi)\,\min\{ 1,|y|^\vartheta\},
\end{eqnarray*}
for some~$C>0$.

Moreover, since~$F'(-\infty)=0$,
\begin{eqnarray*}&&
\Bigg|
F'\left( \frac{\beta(p_n+y_n)-\beta(p_n)}{|y|}-\frac{\varpi \big(\Phi(p'+y')
-\Phi(p')\big)}{|y|}\right)
\Bigg|\\
&=&
\Bigg|
F'\left( \frac{\beta(p_n+y_n)-\beta(p_n)}{|y|}-\frac{\varpi \big(\Phi(p'+y')
-\Phi(p')\big)}{|y|}\right)
\Bigg|\;\chi_{(-p_n-L,-p_n+L)}(y_n)
\\&\le& C\,\chi_{(-2L,2L)}(y_n),
\end{eqnarray*}
up to renaming~$C>0$. 

{F}rom these observations, we deduce that
$$ |\Psi'(\varpi)|\le
C\,(1+\varpi)\,\left(
\min\{ 1,|y|^\vartheta\}+
\frac{|y'|^{1+\vartheta}\chi_{(-2L,2L)}(y_n)}{|y|}
\right),$$
up to renaming~$C>0$.

Consequently, recalling~\eqref{923461d52s99tgv7},
and writing~$E^{(\beta,0)}$ to denote~$E^{(\beta,\varpi)}$ with~$\varpi:=0$,
we find that
\begin{eqnarray*}&&
\big| {\mathcal{H}}^s_{E^{(\beta,\varpi)}}(p)-{\mathcal{H}}^s_{E^{(\beta,0)}}(p)\big|\\
&=&\frac12\,\left|\int_{\R^n}\Big(
\Psi(\varpi)-\Psi(0)\Big)
\,\frac{dy}{|y|^{n+s}}\right|\\
&\le& C\,\varpi\,(1+\varpi)\,
\int_{\R^n}
\left(
\min\{ 1,|y|^\vartheta\}+
\frac{|y'|^{1+\vartheta}\chi_{(-2L,2L)}(y_n)}{|y|}
\right)
\frac{dy}{|y|^{n+s}}\\
&\le& C\,\varpi\,(1+\varpi)\,\left(1+
\int_{\R^n}
|y'|^{1+\vartheta}\chi_{(-2L,2L)}(y_n)\,
\frac{dy}{|y|^{n+1+s}}\right)\\
&\le& C\,\varpi\,(1+\varpi)\,\left(1+
\int_{{\R^n}\atop{\{|y'|>4L\}}}
|y'|^{1+\vartheta}\chi_{(-2L,2L)}(y_n)\,
\frac{dy}{|y|^{n+1+s}}\right)\\
&\le& C\,\varpi\,(1+\varpi)\,\left(1+
\int_{{\R^{n-1}}\atop{\{|y'|>4L\}}}
\frac{dy'}{|y'|^{n-\vartheta+s}}\right)\\&\le& C\,\varpi\,(1+\varpi),
\end{eqnarray*}
up to renaming~$C$ line after line.

The desired result now plainly follows
from the latter inequality and the fact that, in view of
Lemma~\ref{5444542e62734}, we know that~${\mathcal{H}}^s_{E^{(\beta,0)}}(p)=
C(n,s)\,
{\mathcal{H}}^s_{E_\beta}(p_n,p_{n+1})$.
\end{proof}

In the light of Lemma~\ref{VGbjatrtta9}, we can now construct
the following useful barrier:

\begin{lemma}\label{RCAIcgaue023}
Let~$\tilde\ell>0$ and~$\bar\ell\in[-\tilde\ell,\tilde\ell]$. Let~$\lambda>0$ and~$L\in(\lambda,+\infty)$. Let~$
a$, $b$, $c>0$, $\alpha\in(0,s)$.

Let~$\e\in(0,1)$ and assume that
\begin{equation}\label{43266-78j1}
L\ge\frac{c}{\e^{1/s}}.
\end{equation}
Let
$$ \beta(x_n):=\begin{cases}
\bar\ell x_n & {\mbox{ if }}x_n\in(-L,0),\\
\bar\ell x_n+\e a x_n& {\mbox{ if }}x_n\in[0,\lambda],\\
\bar\ell x_n-\e bx_n^{1+\alpha} & {\mbox{ if }}x_n\in(\lambda,L).
\end{cases}$$
Then there exist
\begin{equation}\label{43266-78j2}
\mu\in \left(0,\frac\lambda8\right)\end{equation}
depending only on~$n$,
$s$,
$\tilde\ell$, $\lambda$, $a$, $b$, $c$ and~$\alpha$
(but independent of~$\e$ and~$L$), 
and~$C_\star>0$,
depending only on~$n$,
$s$,
$\tilde\ell$, $\lambda$, $a$, $b$, $c$ and~$\alpha$,
such that, if
\begin{equation}\label{43266-78j3}
\varpi\in \left[0,\min\left\{1,\,\frac{C_\star\,\e}{\mu^s}\right\}\right],\end{equation}
then
\begin{equation} \label{LAOS923i828dgtwqr}
{\mathcal{H}}^s_{E^{(\beta,\varpi)}}(p)\le
-\frac{C^\star\,\e}{\mu^s}
<0\end{equation}
for every~$p=(p_1,\dots,p_n,p_{n+1})\in\partial E^{(\beta,\varpi)}$ with~$p_n\in(0,\mu)$,
where~$C^\star>0$ depends only on~$n$,
$s$,
$\bar\ell$, $\lambda$, $a$, $b$, $c$ and~$\alpha$.

Moreover, if~$E^{(\beta,\varpi,R)}:=E^{(\beta,\varpi)}\cap \{ |x'|<R\}$ and
\begin{equation} \label{LAOS923i828dgtwqr-2}
R\ge \frac{C_\sharp\,\mu}{\e^{\frac1s}}
,\end{equation}
for a suitable~$C_\sharp>0$ depending only on~$n$,
$s$,
$\tilde\ell$, $\lambda$, $a$, $b$, $c$ and~$\alpha$,
then
\begin{equation} \label{LAOS923i828dgtwqr-3}
{\mathcal{H}}^s_{E^{(\beta,\varpi,R)}}(p)<0\end{equation}
for every~$p=(p_1,\dots,p_n,p_{n+1})\in\partial E^{(\beta,\varpi,R)}$ with~$p_n\in(0,\mu)$
and~$|p'|<R/2$.
\end{lemma}

\begin{proof} Firstly, we prove~\eqref{LAOS923i828dgtwqr}.
By Lemma~7.1 in~\cite{2019arXiv190405393D},
we know that, under assumptions~\eqref{43266-78j1} and~\eqref{43266-78j2},
if~$\mu$ is sufficiently small we have that
$$ {\mathcal{H}}^s_{E_\beta}(p_n,p_{n+1})\le -\frac{C'(\tilde\ell,a)\,\e}{\mu^s},
$$
for some~$C'(\tilde\ell,a)>0$.

{F}rom this and Lemma~\ref{VGbjatrtta9}, we deduce that
\begin{eqnarray*}
{\mathcal{H}}^s_{E^{(\beta,\varpi)}}(p)&\le& C(n,s)\,
{\mathcal{H}}^s_{E_\beta}(p_n,p_{n+1})+ C\,\varpi\,(1+\varpi)\\&
\le& -\frac{C(n,s)\,C'(\tilde\ell,a)\,\e}{\mu^s}+ C\,\varpi\,(1+\varpi).\end{eqnarray*}
{F}rom this and~\eqref{43266-78j3} we obtain the desired result
in~\eqref{LAOS923i828dgtwqr}.

Furthermore, the sets~$E^{(\beta,\varpi,R)}$ and~$E^{(\beta,\varpi)}$
coincide in~$\{|x'|<R\}$ and therefore, if additionally~$|p'|\le R/2$,
\begin{eqnarray*} &&{\mathcal{H}}^s_{E^{(\beta,\varpi,R)}}(p)
\le
{\mathcal{H}}^s_{E^{(\beta,\varpi)}}(p)+\int_{\R^{n+1}\cap \{|y'|>R\}
}\frac{dy}{|y-p|^{n+1+s}}
\\&&\qquad \le
{\mathcal{H}}^s_{E^{(\beta,\varpi)}}(p)+\int_{\R^{n+1}\setminus B_{R/2}
}\frac{dz}{|z|^{n+1+s}}
\le -\frac{C^\star\,\e}{\mu^s}+\frac{C}{R^s}\le -\frac{C^\star\,\e}{2\mu^s}
,\end{eqnarray*}
thanks to~\eqref{LAOS923i828dgtwqr}
and~\eqref{LAOS923i828dgtwqr-2}, and this proves~\eqref{LAOS923i828dgtwqr-3}.
\end{proof}

\section{Proof of Theorem~\ref{NLMS}}\label{657699767395535}

The proof of Theorem~\ref{NLMS}
consists in combining Theorem~\ref{QC} (or, more specifically,
Corollary~\ref{COR:ALTE})
with the 
boundary Harnack Inequality and the boundary
improvement of flatness methods
introduced in~\cite{2019arXiv190405393D}.
More specifically, in view of the boundary Harnack Inequality
in~\cite{2019arXiv190405393D},
one can
rephrase Lemma~6.2 of~\cite{2019arXiv190405393D}
in our setting
and obtain the following 
convergence result of the vertical rescalings
to a linearized equation:

\begin{lemma}\label{NOANpierjfppp34}
Let~$\e\in(0,1)$, $\alpha\in(0,s)$ and~$u:\R^2\to\R$. Set
$$ u_\e(x):=\frac{u(x)}\e.$$
There exists~$c_0\in(0,1)$,
depending only on~$\alpha$ and~$s$, such that the following
statement holds true. Assume that~$u$
is an $s$-minimal graph in~$(-2^{\tilde k_0},2^{\tilde k_0})
\times(0,2^{\tilde k_0})$, with
$$ \tilde k_0:=\left\lceil\frac{|\log\e|}{c_0}\right\rceil.$$
Suppose also that
$$
|u(x)|\le\e^{\frac1{c_0}}\,|x|^{1+\alpha}
\qquad{\mbox{for all $x=(x_1,x_2)\in B_{ 2^{\tilde k_0} }$
with~$x_2<0$}}$$
and
$$
|u(x)|\le\e\,(2^k)^{1+\alpha}\qquad{\mbox{for all $x=(x_1,x_2)\in B_{ 2^{k} }$
with~$x_2>0$, for all $k\in\{0,\dots,\tilde k_0\}$.}}
$$
Then, as~$\e\searrow0$, up to a subsequence,
$u_\e$ converges locally uniformly in~$\R^2$ to a function~$\bar u$ satisfying
\begin{eqnarray*}&&
\sup_{x\in\R^2}\frac{|\bar u(x)|}{1+|x|^{1+\alpha}}<+\infty\\
{\mbox{and }}&&(-\Delta)^{\frac{1+s}2}\bar u=0\;\,
{\mbox{ in }}\;\,\R \times(0,+\infty).
\end{eqnarray*}
Furthermore, if~$|x|\to0$ with~$x_2>0$, we can write that
$$ \bar u(x)=\bar a \,x_2^{\frac{1+s}2}+O(|x|^{\frac{3+s}2}),$$
for some~$\bar a\in\R$.
\end{lemma}

With this, one can obtain a suitable
improvement of flatness result as follows:

\begin{theorem}\label{ijsd78AIsjjd}
Let~$\e$, $h\in(0,1)$, $\alpha\in(0,s)$, $u:\R^2\to\R$
be such that
\begin{equation}\label{65676589cs0an3o}
{\mbox{$u(x)=0$ for all~$x\in (-2^{\tilde k_0},2^{\tilde k_0})\times(-h,0)$,}}\end{equation}
and assume that~$u$ 
is an $s$-minimal graph in~$(-2^{\tilde k_0},2^{\tilde k_0})
\times(0,2^{\tilde k_0})$, with
\begin{equation}\label{LAmbielsz} \tilde k_0:=\left\lceil\frac{|\log\e|}{c_0}\right\rceil,\end{equation}
and with~$c_0\in(0,1)$ suitably small.

Suppose that
\begin{equation}\label{ESAy} \lim_{{x\to0}\atop{x_2>0}} u(x)=0.\end{equation}
Then, there exists~$\tilde\e_0\in(0,1)$ such that if~$\e\in(0,\tilde\e_0]$
the following statement holds true.

If
\begin{eqnarray}\label{LAPwkeryuifmma}&&
|u(x)|\le\e^{\frac1{c_0}}\,|x|^{1+\alpha}\quad{\mbox{for all $x\in B_{ 2^{\tilde k_0} }$
with~$x_2<0$}}\end{eqnarray}
and
\begin{eqnarray}
\label{LAPwkeryuifmma2}&&
|u(x)|\le\e\,(2^k)^{1+\alpha}\quad{\mbox{for all $x\in B_{ 2^{k} }$
with~$x_2>0$, for all $k\in\{0,\dots,\tilde k_0\}$,}}
\end{eqnarray}
then, for all~$j\in\N$,
$$ |u(x)|\le \frac\e{2^{j(1+\alpha)}}\qquad{\mbox{ for all $x\in
B_{1/2^j}$ with~$x_2>0$}}.$$
Moreover,
$$ |u(x)|\le4\e|x|^{1+\alpha}
\qquad{\mbox{ for all $x\in
B_{1/2}$ with~$x_2>0$}}.$$
\end{theorem}

\begin{proof} The proof of Theorem~8.1 of~\cite{2019arXiv190405393D}
carries over to this case, with the exception of the
proof of~(8.9) in~\cite{2019arXiv190405393D}
(which in turn uses the one-dimensional
barrier built in Lemma~7.1
of~\cite{2019arXiv190405393D}, that is
not available in the higher dimensional case
that we deal with here).
More precisely, using Lemma~\ref{NOANpierjfppp34},
we obtain that, given~$\delta\in(0,1)$,
if~$\e$ is sufficiently small, for all~$x\in B_3$ with~$x_2>0$,
we have that
\begin{equation}\label{78hs:93e}
|u(x)-\e\,\bar a\,x_2^{\frac{1+s}2}|\le \overline{C}\e\,\big(|x|^{\frac{3+s}2}+\delta\big),\end{equation}
for some~$\overline{C}>0$.

With this, which replaces formula~(8.8)
of~\cite{2019arXiv190405393D} in this setting,
the proof of Theorem~8.1 of~\cite{2019arXiv190405393D}
can be applied to our framework, once we show that
\begin{equation}\label{RUsnsca}
\bar a=0.\end{equation}
To check that~\eqref{RUsnsca} is satisfied, and thus complete
the proof of Theorem~\ref{ijsd78AIsjjd},
we argue for a contradiction and assume, for instance, that~$
\bar a>0$ (the case~$\bar a<0$ being similar).
We take~$E^{(\beta,\varpi,R)}$ as in Lemma~\ref{RCAIcgaue023},
with~$n:=2$,
\begin{eqnarray*}&& 
\Phi(x_1):=|x_1|^{\frac{3+s}{2}},\qquad
\bar\ell:=\e\delta,\qquad
\lambda:=\min\left\{\frac1{10^{\frac{2}{1-s}}},\frac{\bar a}{4\,\overline{C}}\right\},\\&&
c_0:=\min\left\{ s\log 2,\frac{1-\log2}{2},\frac{s}{2(s+1+\alpha)}\right\},\qquad\varpi:=\frac{C_\star\,\e}{\mu^s},\qquad
L:=\frac{1}{4\,\e^{\frac{\log 2}{c_0}}},
\\&&
a:=\frac{\bar a}{8},\qquad b:=2^{2(2+\alpha)}+\frac{1}{\lambda^\alpha}+
\frac{2^{4(2+\alpha)}}{\lambda^{1+\alpha}},
\qquad c:=\frac{1}{4}\qquad{\mbox{ and }}\qquad
R:=\frac{C_\sharp\,\mu}{\e^{\frac1s}},
\end{eqnarray*}
and we slide it from below till it touches the graph of~$u$
(by choosing conveniently the free parameters such that~$\e\ll \delta\ll\mu\ll1$).
As a matter of fact, we have that~\eqref{43266-78j1},
\eqref{43266-78j3}
and~\eqref{LAOS923i828dgtwqr-2} are satisfied.
Consequently, we are in the position of using Lemma~\ref{RCAIcgaue023}
and deduce from~\eqref{LAOS923i828dgtwqr-3} that
\begin{equation}\label{MPLE1}
\begin{split}
{\mbox{if~$E^{(\beta,\varpi,R)}\setminus{\mathcal{D}}$
is below the graph of~$u$,}}\\
{\mbox{then~$E^{(\beta,\varpi,R)}$
is below the graph of~$u$.}}
\end{split}
\end{equation}
where~${\mathcal{D}}:=\{|x_1|<R/2\}\cap\{x_2\in(0,\mu)\}$.

Now we claim that
\begin{equation}\label{MPLE2}
{\mbox{$E^{(\beta,\varpi,R)}\setminus{\mathcal{D}}$
lies below the graph of~$u$.}}
\end{equation}
Indeed, if~$x_2\le-L$, or~$x_2\ge L$, or~$|x_1|\ge R$, then the result
is obvious. Hence, to prove~\eqref{MPLE2}, we can focus
on the region~$\{|x_1|<R\}\times\{|x_2|<L\}$.

If~$x_2\in(-L,0)$, we notice that
\begin{equation}\label{Kytilsq} 2^{\tilde k_0}\ge 2^{\frac{|\log\e|}{c_0}-1}=2L,
\end{equation}
thanks to~\eqref{LAmbielsz}. As a consequence, $|x|\le|x_1|+|x_2|<
R+L<2L\le 2^{\tilde k_0}$, and thus
we can
exploit~\eqref{LAPwkeryuifmma} and find that
\begin{equation*}
|u(x)|\le \e^{\frac1{c_0}}|x|^{1+\alpha}.
\end{equation*}
Combining this with~\eqref{65676589cs0an3o}, and noticing
that~$|x_1|<R=\frac{C_\sharp\,\mu}{\e^{\frac1s}}<2^{\tilde k_0}$,
we see that
$$ |u(x)|\le\e^{\frac1{c_0}}|x|^{1+\alpha}\chi_{(-\infty,h)}(x_2),$$
and, as a result,
\begin{equation}\label{9423331236d65}
\begin{split}
&\beta(x_2)-\varpi\,\Phi(x_1)-u(x)\le
\e\delta x_2-\varpi\,\Phi(x_1)+\e^{\frac1{c_0}}|x|^{1+\alpha}\chi_{(-\infty,h)}(x_2)\\
&\qquad\le- \e\delta |x_2|-\varpi\,\Phi(x_1)+
2^{\frac{1+\alpha}2}\e^{\frac1{c_0}}\big(|x_1|^{1+\alpha}\chi_{(-\infty,h)}(x_2)+|x_2|^{1+\alpha}\big)\\
&\qquad\le \e |x_2|\big( 2^{\frac{1+\alpha}2}L^\alpha\e^{\frac{1-c_0}{c_0}}-\delta\big)
-\varpi\,\Phi(x_1)+
2^{\frac{1+\alpha}2}\e^{\frac1{c_0}} |x_1|^{1+\alpha}\chi_{(-\infty,h)}(x_2)\\
&\qquad= \e |x_2|\left( 2^{\frac{1+\alpha}2}\left(
\frac{1}{4\e^{\frac{\log2}{c_0}}}\right)^\alpha\e^{\frac{1-c_0}{c_0}}-\delta\right)
-\varpi\,\Phi(x_1)+
2^{\frac{1+\alpha}2}\e^{\frac1{c_0}} |x_1|^{1+\alpha}\chi_{(-\infty,h)}(x_2)
\\&\qquad\leq-\frac{\delta\e}2 |x_2|-\varpi\,\Phi(x_1)+
2^{\frac{1+\alpha}2}\e^{\frac1{c_0}} |x_1|^{1+\alpha}\chi_{(-\infty,h)}(x_2).
\end{split}\end{equation}
Now we consider two regimes: when~$|x_2|\le h$, we deduce from~\eqref{9423331236d65}
that
\begin{equation}\label{9423331236d65:2}
\begin{split}
&\beta(x_2)-\varpi\,\Phi(x_1)-u(x)\le 
-\frac{\delta\e}2 |x_2|<0.
\end{split}\end{equation}
If instead~$|x_2|>h$,
using~\eqref{9423331236d65} we infer that
\begin{equation}\label{9423331236d65:3}
\begin{split}
&\beta(x_2)-\varpi\,\Phi(x_1)-u(x)\le
-\frac{\delta\e}2 h+
2^{\frac{1+\alpha}2}\e^{\frac1{c_0}} |x_1|^{1+\alpha}\\
&\qquad\le- \e\left(
\frac{\delta}2 h-
2^{\frac{1+\alpha}2}\e^{\frac{1-c_0}{c_0}}R^{1+\alpha}\right)=-
\e\left(
\frac{\delta}2 h-
2^{\frac{1+\alpha}2}
C_\sharp^{1+\alpha}\,\mu^{1+\alpha}
\e^{\frac{1-c_0}{c_0}-
\frac{1+\alpha}s
}\right)\\&\qquad\le-\frac{\delta\e h}{4}<0.
\end{split}\end{equation}
In view of~\eqref{9423331236d65:2} and~\eqref{9423331236d65:3},
we conclude that~\eqref{MPLE2}
is satisfied when~$x_2\in(-L,0)$ and~$|x_1|<R$.

{F}rom these considerations, we see that
it is enough now to check~\eqref{MPLE2}
with~$x_2\in(0,\mu)$ and~$|x_1|\in(R/2,R)$,
and with~$x_2\in(\mu,L)$ and~$|x_1|<R$.

If~$x_2\in(0,\mu)$ and~$|x_1|\in(R/2,R)$, in light of~\eqref{78hs:93e}
we have that
\begin{eqnarray*}
&&\beta(x_2)-\varpi\,\Phi(x_1)-u(x)\le
\e\delta x_2+\e a x_2 -\varpi\,\Phi(x_1)-
\e\,\bar a\, x_2^{\frac{1+s}2}+ \overline{C}\e\,\big(|x|^{\frac{3+s}2}+\delta\big)
\\&&\qquad
\le
\e\delta x_2+\e a x_2 -\varpi\,\Phi(x_1)-
\e\,\bar a\,x_2^{\frac{1+s}2}+ \overline{C}\e\delta+
2^{\frac{3+s}4}\,
\overline{C}\e\,\big(x_2^{\frac{3+s}2}+|x_1|^{\frac{3+s}2}\big)\\&&
\qquad
\le \e\,x_2^{\frac{1+s}2}\big( (\delta+a)\mu^{\frac{1-s}2}+
2^{\frac{3+s}4}\,\overline{C}\mu
-\bar a\big) 
+ \overline{C}\e\delta-\big(\varpi\,\Phi(x_1)-
2^{\frac{3+s}4}\,
\overline{C}\e\, |x_1|^{\frac{3+s}2}\big)\\&&
\qquad
\le \e\left(\overline{C}\delta-\left(\frac{C_\star}{\mu^s}\,\Phi(x_1)-
2^{\frac{3+s}4}\,
\overline{C}\, |x_1|^{\frac{3+s}2}\right)\right)\\&&
\qquad
\le \e\left(\overline{C}\delta-\left(\frac{C_\star}{\mu^s}-
2^{\frac{3+s}4}\,
\overline{C}\right)|x_1|^{\frac{3+s}2}\right)
\\&&
\qquad
\le \e\left(\overline{C}\delta-\frac{C_\star}{2\mu^s}\,\left(\frac{R}2\right)^{\frac{3+s}2}\right)<0.
\end{eqnarray*}
This checks~\eqref{MPLE2}
when~$x_2\in(0,\mu)$ and~$|x_1|\in(R/2,R)$,
and therefore we are only left with the
case in which~$x_2\in(\mu,L)$ and~$|x_1|<R$.

In this setting, we distinguish when~$x_2\in(\mu,\lambda)$ and~$|x_1|<R$
and when~$
x_2\in(\lambda,L)$ and~$|x_1|<R$.

Then,
when~$x_2\in(\mu,\lambda)$ and~$|x_1|<R$, we make use of~\eqref{78hs:93e}
and we see that
\begin{eqnarray*}
&&\beta(x_2)-\varpi\,\Phi(x_1)-u(x)\le
\e\delta x_2+\e a x_2-\varpi\,\Phi(x_1)
-
\e\,\bar a\, x_2^{\frac{1+s}2}+ \overline{C}\e\,\big(|x|^{\frac{3+s}2}+\delta\big)\\&&
\qquad
\le
\e(\delta + a) x_2-\varpi\,\Phi(x_1)
-
\e\,\bar a\, x_2^{\frac{1+s}2}+ 2^{\frac{3+s}4}\overline{C}\e\,\big(
|x_1|^{\frac{3+s}2}+x_2^{\frac{3+s}2}\big)+
\overline{C}\e\delta
\\&&
\qquad
=\e x_2^{\frac{1+s}2}\left(
(\delta + a) x_2^{\frac{1-s}2}+
2^{\frac{3+s}4}\overline{C}x_2
+\frac{\overline{C}\delta}{
x_2^{\frac{1+s}2}}
-\bar a \right)+ 2^{\frac{3+s}4}\overline{C}\e\,
|x_1|^{\frac{3+s}2}
-\varpi\,\Phi(x_1)\\&&
\qquad
\le\e x_2^{\frac{1+s}2}\left(
(\delta + a) \lambda^{\frac{1-s}2}+
2^{\frac{3+s}4}\overline{C}\lambda
+\frac{\overline{C}\delta}{
\mu^{\frac{1+s}2}}
-\bar a \right)+ \e|x_1|^{\frac{3+s}2}\left(2^{\frac{3+s}4}\overline{C}\,
-\frac{C_\star}{\mu^s}\right)\\
&&
\qquad
\le-\frac{\e \bar a}{16} x_2^{\frac{1+s}2}<0,
\end{eqnarray*}
that gives~\eqref{MPLE2} in this case.

If instead~$
x_2\in(\lambda,L)$ and~$|x_1|<R$, we exploit~\eqref{LAPwkeryuifmma2}.
In this case, we claim that
\begin{equation}\label{8UAll12-39jjdna}
|u(x)|\le 2^{1+\alpha}\e\max\{1,|x|^{1+\alpha}\}.
\end{equation}
Indeed, if~$|x|<1$, we use~\eqref{LAPwkeryuifmma2} with~$k:=0$
and we obtain~\eqref{8UAll12-39jjdna}.
If instead~$|x|\ge1$, we remark that~$|x|\le|x_1|+|x_2|<R+L<2^{\tilde k_0}$,
and therefore we are in the position of applying~\eqref{LAPwkeryuifmma2}
with~$k$ such that~$2^{k-1}\le |x|\le 2^k$, and conclude that
$$ |u(x)|\le\e\,(2^k)^{1+\alpha}\le 2^{1+\alpha}\e\,|x|^{1+\alpha}.$$
This completes the proof of~\eqref{8UAll12-39jjdna}.

Consequently, using~\eqref{8UAll12-39jjdna}, we find that
\begin{equation}\label{ABSbaatiina}
\begin{split}
&\beta(x_2)-\varpi\,\Phi(x_1)-u(x)\le\e\delta x_2-\e b x_2^{1+\alpha}-\varpi\,\Phi(x_1)+
2^{1+\alpha}\e\max\{1,|x|^{1+\alpha}\}\\
&\qquad\le \e\delta x_2-\e b x_2^{1+\alpha}-\varpi\,\Phi(x_1)+
2^{2(1+\alpha)}\e\,\big(1+|x_1|^{1+\alpha}+x_2^{{1+\alpha}}\big)\\
&\qquad\le \e x_2^{1+\alpha}
\left(\frac\delta{x_2^\alpha}+
2^{2(1+\alpha)}- b 
\right)+2^{2(1+\alpha)}\e+\e\left(
2^{2(1+\alpha)} |x_1|^{1+\alpha}-\frac{C_\star}{\mu^s}\Phi(x_1)
\right).
\end{split}\end{equation}
Now we consider the function
$$ [0,+\infty)\ni t\mapsto g(t):= t^{1+\alpha}-K t^{\frac{3+s}2},$$
for a given~$K>0$, and we claim that
\begin{equation}\label{y7so239dj}
g(t)\le \frac{1}{K^{\frac{2(1+\alpha)}{1+s-2\alpha}}}.
\end{equation}
Indeed, if~$t\ge\frac{1}{K^{\frac{2}{1+s-2\alpha}}}$,
we see that
$$ g(t)=t^{1+\alpha}\left(1-K t^{\frac{1+s-2\alpha}2}\right) \le0,$$
and so~\eqref{y7so239dj} is satisfied. If instead~$t\in\left[0,\frac{1}{K^{\frac{2}{1+s-2\alpha}}}\right)$
we have that
$$ g(t)\le t^{1+\alpha}\le\frac{1}{K^{\frac{2(1+\alpha)}{1+s-2\alpha}}},$$
proving~\eqref{y7so239dj} also in this case.

Now, from~\eqref{y7so239dj},
\begin{eqnarray*}&&
2^{2(1+\alpha)} |x_1|^{1+\alpha}-\frac{C_\star}{\mu^s}\Phi(x_1)=
2^{2(1+\alpha)} \left(|x_1|^{1+\alpha}-\frac{C_\star}{2^{2(1+\alpha)}\mu^s}
|x_1|^{\frac{3+s}2}\right)\\
&&\qquad\le 2^{2(1+\alpha)}\left( \frac{2^{2(1+\alpha)}\mu^s}{C_\star}\right)^{\frac{2(1+\alpha)}{1+s-2\alpha}}
= \tilde C\,\mu^{\frac{2s(1+\alpha)}{1+s-2\alpha}},
\end{eqnarray*}
for a suitable~$\tilde C>0$.

Plugging this information into~\eqref{ABSbaatiina}, we conclude that
\begin{eqnarray*}&&
\beta(x_2)-\varpi\,\Phi(x_1)-u(x)\le
\e x_2^{1+\alpha}
\left(\frac\delta{x_2^\alpha}+
2^{2(1+\alpha)}- b 
\right)+2^{2(1+\alpha)}\e+\tilde C\,\e\mu^{\frac{2s(1+\alpha)}{1+s-2\alpha}}\\&&\qquad
\le
\e x_2^{1+\alpha}
\left(\frac\delta{\lambda^\alpha}+
2^{2(1+\alpha)}- b 
\right)+2^{2(2+\alpha)}\e
\le
-\frac{2^{4(2+\alpha)}\e x_2^{1+\alpha}}{\lambda^{1+\alpha}}+2^{2(2+\alpha)}\e\\&&\qquad
\le
-2^{4(2+\alpha)}\e +2^{2(2+\alpha)}\e<0,
\end{eqnarray*}
and this completes the proof of~\eqref{MPLE2}.

Then, \eqref{MPLE1} and~\eqref{MPLE2}
give that~$E^{(\beta,\varpi,R)}$
is below the graph of~$u$.

As a result, if~$x_2\in(0,\lambda)$ and~$|x_1|<R/2$,
$$ \e\delta x_2+\frac{\e\bar a}8 x_2-\varpi\Phi(x_1)
=
\bar\ell x_2+\e a x_2-\varpi\Phi(x_1)=\beta(x)-\varpi\Phi(x_1)\le u(x).$$
Consequently, if we consider the second blow-up of the graph of~$u$,
as in Lemma~\ref{BOAmdfiUAMP}, we conclude that
$$ \left\{x_{3}< \left(\e\delta +\frac{\e\bar a}8\right) x_2
\right\}\cap \{x_2>0\} \subseteq E_{00}.$$
{F}rom this and Corollary~\ref{COR:ALTE}, we conclude that~\eqref{POS-E001-BI1b}
and~\eqref{POS-E001-BI3b} cannot hold true,
and therefore necessarily~\eqref{POS-E001-BI2b} must be satisfied,
namely~$E_{00}\cap\{x_2>0\}=\{x_2>0\}$
and 
$$ \lim_{x_2\searrow0} u(0,x_2)>0.$$
The latter inequality is in contradiction with~\eqref{ESAy}
and therefore this completes the proof of the desired claim in~\eqref{RUsnsca}.
\end{proof}

Now, we can complete the proof of Theorem~\ref{NLMS} :

\begin{proof}[Proof of Theorem~\ref{NLMS} ]
The core of the proof is to establish~\eqref{VANP},
since~\eqref{THAN-2} would then easily follow.

To this end, from Theorem~\ref{ijsd78AIsjjd}, 
arguing as in Theorem~8.2 of~\cite{2019arXiv190405393D},
and exploiting Corollary~\ref{COR:ALTE} here
(instead of Theorem~4.1 in~\cite{2019arXiv190405393D}) to obtain~(8.19)
in~\cite{2019arXiv190405393D},
one establishes~\eqref{VANP}
with~$\frac{3+s}2$ replaced by~$1+\alpha$,
for a given~$\alpha\in(0,s)$. 

Then, to improve this
regularity exponent and complete the proof of~\eqref{VANP}, one can proceed as in the proof of
Theorem~1.2 of~\cite{2019arXiv190405393D}.
\end{proof}

\begin{appendix}

\section{A classical counterpart of Theorem~\ref{FLFL}}\label{APPA}

In this appendix, we show that Theorem~\ref{FLFL}
possesses a classical counterpart for the Laplace equation, namely:

\begin{theorem}\label{FLFLCLA}
Let~$n\ge2$, $k\in\N$ and~$f\in C^k(\overline{B_1'})$. Then, for every~$\e>0$
there exist~$f_\e\in C^k(\R^{n-1})$ and~$u_\e\in C(\R^n)$ such that
\begin{equation*} \begin{cases}
\Delta u_\e=0 & {\mbox{ in }} B_1\cap \{x_n>0\},\\
u_\e=0 & {\mbox{ in }} B_1\cap\{x_n=0\},\end{cases}\end{equation*}
\begin{equation*}
\displaystyle\lim_{x_n\searrow0} \frac{u_\e(x',x_n)}{x_n}=f_\e(x')\qquad{\mbox{ for all }}
x'\in B_1',\end{equation*}
and
\begin{equation*}
\| f_\e-f\|_{C^k(\overline{B_1'})}\le\e.
\end{equation*}
\end{theorem}

We observe that formally Theorem~\ref{FLFLCLA} corresponds to Theorem~\ref{FLFL}
with~$\sigma:=1$. The arguments that we proposed for~$\sigma\in(0,1)$
can be carried out also when~$\sigma=1$, and thus prove Theorem~\ref{FLFLCLA}.
Indeed, for the classical Laplace equation,
one can obtain Lemma~\ref{BAN}
by taking, for instance, the real part of the holomorphic function~$\C\ni z\mapsto z^k$, with~$k\in\N$, and then use
this homogeneous solution in the proof of Lemma~\ref{fan}.
Then, once Lemma~\ref{fan} is proved with~$\sigma:=1$, one can
exploit the proof of 
Theorem~\ref{FLFL} presented on page~\pageref{89-ppA}
with~$\sigma:=1$ and obtain Theorem~\ref{FLFLCLA}.

However, in the classical case there is also an explicit polynomial expansion
which recovers Lemma~\ref{fan} with~$\sigma:=1$, that is when~\eqref{KK}
is replaced by
\begin{equation*} \begin{cases}
\Delta u=0 & {\mbox{ in }} B_r\cap \{x_n>0\},\\
u=0 & {\mbox{ in }} B_r\cap\{x_n=0\}.\end{cases}\end{equation*}

Hence, to establish Theorem~\ref{FLFLCLA}, we focus on the following argument
of classical flavor (which is not reproducible for~$\sigma\in(0,1)$
and thus requires new strategies in the case of Theorem~\ref{FLFL}).

\begin{proof}[Proof of Lemma~\ref{fan} with~$\sigma:=1$]
Given~$\kappa=\big(\kappa_\gamma
\big)_{{\gamma\in\N^{n-1}}\atop{|\gamma|\le k}}$ in~$\R^{N_k}$, where~$N_k$ is defined
in~\eqref{ennekappa},
we aim at finding a function~$u\in C(\overline{B_1})$
such that
\begin{equation}\label{QLA} \begin{cases}
\Delta u=0 & {\mbox{ in }} B_1\cap \{x_n>0\},\\
u=0 & {\mbox{ in }} B_1\cap\{x_n=0\},\end{cases}\end{equation}
and for which there exists~$\phi\in C^k(\overline{B'_1})$ such that
\begin{equation}\label{uvA4}
\lim_{x_n\searrow0} \frac{u(x',x_n)}{x_n}=\phi(x')\qquad{\mbox{ for all }}
x'\in B_1',\end{equation}
with, in the notation of~\eqref{mul},
\begin{equation}\label{78ssdJA}
{\mathcal{D}}^k\phi(0)=\kappa .\end{equation}
To this end, we define~$\phi$ to be the polynomial
$$ \phi(x'):= \sum_{{\gamma\in\N^{n-1}}\atop{|\gamma|\le k}}
\frac{\kappa_\gamma}{\gamma!} (x')^\gamma.$$
In this way, \eqref{78ssdJA} is automatically satisfied. Furthermore, we observe that~$\Delta_{x'}^j\phi$
vanishes identically when~$2j>k$, and therefore we can set
$$ u(x',x_n):=\sum_{{j\in\N}\atop{j\le k/2}} \frac{(-1)^j\,
\Delta_{x'}^j \phi(x')}{(2j+1)!}\,x_n^{2j+1}
=
\sum_{j=0}^{+\infty} \frac{(-1)^j\,
\Delta_{x'}^j \phi(x')}{(2j+1)!}\,x_n^{2j+1}.$$
We observe that
$$\lim_{x_n\searrow0} \frac{u(x',x_n)}{x_n}=
\lim_{x_n\searrow0}
\sum_{{j\in\N}\atop{j\le k/2}} \frac{(-1)^j\,
\Delta_{x'}^j \phi(x')}{(2j+1)!}\,x_n^{2j}=
\phi(x'),$$
which establishes~\eqref{uvA4}.

Also, we see that~$u(x',0)=0$, and moreover,
using the notation~$J:=j-1$,
\begin{eqnarray*}\Delta u(x)&=&\Delta_{x'} u(x)+\partial^2_{x_n}u(x)\\
&=&
\sum_{j=0}^{+\infty} \frac{(-1)^j\,
\Delta_{x'}^{j+1} \phi(x')}{(2j+1)!}\,x_n^{2j+1}
+
\sum_{j=1}^{+\infty} \frac{(-1)^j\,
\Delta_{x'}^j \phi(x')}{(2j-1)!}\,x_n^{2j-1}\\
&=&
\sum_{j=0}^{+\infty} \frac{(-1)^j\,
\Delta_{x'}^{j+1} \phi(x')}{(2j+1)!}\,x_n^{2j+1}
+
\sum_{J=0}^{+\infty} \frac{(-1)^{J+1}\,
\Delta_{x'}^{J+1} \phi(x')}{(2J+1)!}\,x_n^{2J+1}\\&=&0,
\end{eqnarray*}
which gives~\eqref{QLA}.
\end{proof}

\end{appendix}

\section*{Acknowledgments}

The first and third authors are member of INdAM and
are supported by the Australian Research Council
Discovery Project DP170104880 NEW ``Nonlocal Equations at Work''.
The first author's visit to Columbia has been partially funded by
the Fulbright Foundation
and the Australian Research Council DECRA DE180100957
``PDEs, free boundaries and applications''. The second author is supported
by the National Science Foundation grant DMS-1500438.
The third author's visit to Columbia has been partially funded by
the Italian Piano di Sostegno alla Ricerca 
``Equazioni nonlocali e applicazioni''.

\end{document}